\documentclass[11pt,reqno]{amsart}
\usepackage{newlfont,amsmath,amsthm,amsgen,amscd}
\usepackage{amssymb}
\usepackage{MnSymbol}
\usepackage{amsmath}
\usepackage{latexsym}
\usepackage{fancyhdr}
\usepackage{ifthen}
\usepackage{slashed}
\usepackage{amsfonts}
\usepackage{lastpage}
\usepackage{afterpage}
\usepackage{epsfig}
\usepackage{enumerate}
\usepackage{euscript,color}

\def\blu{\color[rgb]{0,0,1}}

\newcommand{\belabel}[1]{\begin{equation}\label{#1}}

\newcommand{\C}{{\mathbb C}}

\newcommand{\rr}{\mathbb{R}}

\newcommand{\fco}{\mathfrak{co}}
\newcommand{\fg}{\mathfrak g}

\newcommand{\fh}{\mathfrak h}

\newcommand{\fp}{\mathfrak p}

\newcommand{\fhol}{\mathfrak{hol}}
\newcommand{\fgl}{\mathfrak{gl}}

\newcommand{\fsl}{{\mathfrak{sl}}}

\newcommand{\fso}{{\mathfrak{so}}}
\newcommand{\fsp}{{\mathfrak{sp}}}
\newcommand{\fspin}{{\mathfrak{spin}}}
\newcommand{\fsu}{{\mathfrak{su}}}
\newcommand{\hol}{\mathfrak{hol}}

\newcommand{\R}{{\mathrm{R}}}

\newcommand{\tr}{\mathrm{tr}}

\newcommand{\Id}{{\mathrm{Id}}}

\newcommand{\beq}{\begin{eqnarray*}}
\newcommand{\eeq}{\end{eqnarray*}}
\newcommand{\be}{\begin{eqnarray}}
\newcommand{\ee}{\end{eqnarray}}
\newcommand{\Ric}{{\mathrm{Ric}}}
\newcommand{\Sch}{{\mathsf{P}}^g}
\newcommand{\beqn}{\begin{equation}}
\newcommand{\eeqn}{\end{equation}}
\newcommand{\aM}{\widetilde{M}}
\newcommand{\ag}{\widetilde{g}}

\newcommand{\aR}{\widetilde{\R}}
\theoremstyle{definition}

\newtheorem{re}{Remark}[section]

\newtheorem{bsp}{Example}[section]

\newtheorem*{bsp*}{Example}
\newtheorem*{def*}{Definition}
\theoremstyle{plain}
\newtheorem{Lemma}{Lemma}[section]
\newtheorem*{lem*}{Lemma}
\newtheorem{Proposition}{Proposition}[section]
\newtheorem{Corollary}{Corollary}[section]
\newtheorem{Theorem}{Theorem}[section]
\newtheorem*{theo*}{Theorem}
\newtheorem*{conj*}{Conjecture}

\newcommand{\Hol}{\mathrm{Hol}}
\newcommand{\rk}{\mathrm{rk}}

\renewcommand{\aa}{{\bar{a}}}
\newcommand{\bb}{{\bar{b}}}

\newcommand{\s}{\mathrm{s}}
\newcommand{\+}{\oplus}

\newcommand{\Spin}{\mathbf{Spin}}
\newcommand{\spin}{\mathfrak{spin}}

\newcommand{\SO}{\mathbf{SO}}
\newcommand{\so}{\mathfrak{so}}

\newcommand{\Sp}{\mathbf{Sp}}
\renewcommand{\sp}{\mathfrak{sp}}
\newcommand{\csp}{\mathfrak{csp}}

\newcommand{\1}{\mathbf{1}}
\newcommand{\J}{\mathbf{J}}

\newcommand{\g}{\mathbf{g}}

\newcommand{\scal}{\mathrm{scal}}
\newcommand{\del}{\partial}

\newcommand{\e}{\mathrm{e}}

\newcommand{\bleml}[1]{\begin{Lemma} \label{#1}}
\newcommand{\blem}{\begin{Lemma}}
\newcommand{\elem}{\end{Lemma}}

\newcommand{\btheo}{\begin{Theorem}}
\newcommand{\btheol}[1]{\begin{Theorem}\label{#1}}
\newcommand{\etheo}{\end{Theorem}}

\newcommand{\bpropl}[1]{\begin{Proposition} \label{#1}}
\newcommand{\bprop}{\begin{Proposition}}
\newcommand{\eprop}{\end{Proposition}}

\newcommand{\bcorl}[1]{\begin{Corollary} \label{#1}}
\newcommand{\bcor}{\begin{Corollary}}
\newcommand{\ecor}{\end{Corollary}}

\newcommand{\bbem}{\begin{re}}
\newcommand{\ebem}{\end{re}}

\newcommand{\bprf}{\begin{proof}}
\newcommand{\eprf}{\end{proof}}

\begin{document}
\title[The ambient obstruction tensor and conformal holonomy]{The ambient obstruction tensor and conformal holonomy}

\author{Thomas Leistner}\address[TL]{School of Mathematical Sciences, University of Adelaide, SA 5005, Australia} \email{\tt thomas.leistner@adelaide.edu.au}
\author{Andree Lischewski}  
\address[AL]{Institut f\"{u}r Mathematik, Humboldt-University Berlin,
Unter den Linden 6,
10099 Berlin, Germany} \email{\tt  lischews@math.hu-berlin.de}
\thanks{This research was supported by the Australian Research
Council via the grants FT110100429 and DP120104582.}
\date{\today}

\begin{abstract}
For a conformal manifold, we  describe a new relation between the ambient obstruction tensor of Fefferman and Graham and the holonomy of the normal conformal Cartan connection. This relation allows us to prove several  results on the vanishing and the rank of the obstruction tensor, for example  for conformal structures admitting twistor spinors or normal conformal Killing forms. As our main tool we introduce the notion of a conformal holonomy distribution and show that its integrability is closely related to the exceptional  conformal structures in dimensions five and six that  were  found by Nurowski and Bryant. 
\end{abstract}

\maketitle


\section{Introduction}
A conformal structure  of signature $(p,q)$ on a smooth manifold $M$ is an equivalence class $c$ of semi-Riemannian metrics on $M$ of signature $(p,q)$, where two metrics $g$ and $\hat g$ are equivalent if  
$\hat g=\mathrm{e}^{2f}g$ for a smooth function $f$. For conformal structures the construction of local invariants is more complicated than for semi-Riemannian structures, where all local invariants can be derived from the Levi-Civita connection and its curvature. 
For conformal geometry, essentially there are two invariant constructions: the conformal ambient metric of Fefferman and Graham \cite{fefferman/graham85,fefferman-graham07} and the normal conformal Cartan connection \cite{cartan23} with the induced tractor calculus \cite{bailey-eastwood-gover94}. In the present article we will investigate a new relationship between two essential ingredients of these invariant constructions, the {\em ambient obstruction tensor} on one hand, and the {\em conformal holonomy} on the other. We will briefly introduce these notions:

The ambient metric construction assigns to any conformal manifold $(M,[g])$ of signature $(p,q)$ and dimension $n$  a pseudo-Riemannian metric $\ag$ on an open neighbourhood $\aM$ of $Q=M\times \rr^{>0}$ in $\rr\times Q$ with specific properties that link $[g]$ and  $\ag$ as closely as possible. More precisely, 
denoting the the coordinates on $\rr^{>0}$ and $\rr$ by $t$ and $\rho$, respectively, 
$\ag$ is required to restrict to $t^2g$ along $Q$ and moreover its Ricci-tensor vanishes along $Q$ to infinite order in $\rho$ when $n$ is odd and  to order $\rho^{\frac{n}{2}-1}$ when $n$ is even. The seminal result in \cite{fefferman/graham85,fefferman-graham07} is that for  smooth conformal structures, such an ambient metric always exists and is unique to all orders for $n$ odd and up to order $\rho^{\frac{n}{2}-1}$ when $n$ is even. Moreover, in even dimensions the existence of an  ambient metric whose Ricci tensor vanishes along $Q$ to all orders is closely related to the vanishing of a certain symmetric, divergence-free and conformally covariant $(0,2)$-tensor $\cal O$ on $M$, the {\em Fefferman-Graham obstruction tensor} or {\em ambient obstruction tensor}. In four dimensions the obstruction tensor is given by the well known Bach tensor, but in general even dimension no general explicit formula for $\cal O$ exists.  The obstruction tensor will be in the focus of the present article. 

The other invariant construction in conformal geometry is the {\em normal conformal Cartan connection}. This is an $\so(p+1,q+1)$-valued Cartan connection defined on a $P$-bundle, where $P$ is the parabolic subgroup defined by the stabilizer in $\mathbf{O}(p+1,q+1)$ of a lightlike  line in $\rr^{p+1,q+1}$, and it satisfies a certain normalisation condition that defines it uniquely. The normal conformal Cartan connection defines a covariant derivative $\nabla^{nc}$ on a vector bundle $\cal T$, the {\em  conformal tractor connection}  on the {\em standard tractor bundle}. To $(\cal T, \nabla^{nc})$ one can associate the holonomy group  of $\nabla^{nc}$-parallel transports along loops based at $x\in M$. As this group only depends on the conformal structure, it is denoted by $\Hol_x(M,c)$ and called the {\em conformal holonomy}. It is contained in $\mathbf{O}(p+1,q+1)$ and its Lie algebra is denoted by $\fhol_x(M,c)\subset \fso(p+1,q+1)$.

Many interesting conformal structures are related to conformal {\em holonomy reductions}, i.e., conformal structures for which the  conformal holonomy algebra  is a proper subalgebra  of $\fso(p+1,q+1)$. Examples are manifolds admitting twistor spinors, for which the spin representation of the conformal holonomy group admits an invariant spinor. This includes  conformal Fefferman spaces that are closely related to CR-geometry \cite{Fefferman76}, and for which the conformal holonomy reduces to the special unitary group. Another fascinating example are the conformal structures that are determined by generic distributions of rank $2$ in dimension $5$. Such distributions played an important role in the history of the 
simple Lie algebra with exceptional root system $G_2$: 
 Cartan  \cite{cartan93} discovered that for some of these distributions the Lie algebra of symmetries is given by the non-compact exceptional Lie algebra $\fg_2 $ of  type $G_2$. Related to the equivalence problem for such distributions, Cartan constructed the corresponding $\fg_2$-valued  Cartan connection \cite{cartan10}. It was then realised by Nurowski \cite{nurowski04} that to any such distribution one can associate a conformal structure of signature $(2,3)$ whose conformal holonomy is reduced from $\fso(3,4)$ to $\fg_2$. Similarly, Bryant associated to any generic rank $3$ distribution in dimension $6$ a conformal structure of signature $(3,3)$  whose holonomy reduces to $\spin(3,4)\subset\so(4,4)$. Both, and in particular the latter will be relevant in the present article.

The ambient metric construction and the normal conformal Cartan connection turn out to be closely related. Indeed, in \cite{cap-gover03}   tractor data are formulated entirely in terms of ambient data, and in \cite{gope} the  ambient curvature tensors are rewritten in in terms of tractor curvature and derivatives thereof.
The main result  in our paper reveals another interesting correspondence, now between the  the ambient obstruction tensor $\cal O$ and the conformal holonomy. We show that  that the image of  $\cal O$, when    considered as a $(1,1)$-tensor, can be identified with a distinguished subspace in the conformal holonomy algebra $\fhol_x(M,c)$. To be more precise, recall that the Lie algebra $\fso(p+1,q+1)$ is $|1|$-graded as $\fso(p+1,q+1)=\fg_{-1}\+\fg_0\+\fg_{1}$, where $\g_0\simeq \fco(p,q)
$ is the conformal Lie algebra and $\fg_0\+\fg_{1}=\fp$ is the Lie algebra of the parabolic subgroup $P$. It is important to note that $\fg_1$ can be identified with $\rr^{p,q}$ and hence with  the tangent space $T_xM$.
This allows us to prove the following theorem:
\begin{Theorem}\label{maintheo}
Let $(M^{p,q},c)$ be a {smooth} conformal manifold of even dimension $n$ and with ambient obstruction tensor $\mathcal{O}$. Then  the image of $\mathcal{O}$ at $x\in M$ is contained in $\mathfrak{hol}_x(M,c)\cap \fg_1$. In particular, the  rank of $\mathcal{O}$ at each point is limited by the dimension of $\mathfrak{hol}_x(M,c) \cap \mathfrak{g}_1$.
Moreover, 
if $\hol(M,c)$ is a proper subalgebra  of $\fso(p+1,q+1)$ , then the image of $\mathcal{O}$ is totally lightlike. In particular,  $\rk(\mathcal{O}) \leq \text{min}(p,q)$.\end{Theorem} 
The implications  of this result are evident: 
One the one hand it shows that if the obstruction tensor has maximal rank $n$ at some point, then the holonomy is generic. Hence, 
 $\mathcal{O}$ can  be interpreted as a universal obstruction to the existence of parallel tractors on $(M,c)$ of any type. Namely for such a tractor to exist, $\mathcal{O}$ needs to have a nontrivial kernel everywhere. 
On the other hand, 
conformal holonomy reductions can be used to restrict the rank of the obstruction tensor. For example, it is well known that the existence of a parallel standard tractor (and hence of a local Einstein metric in $c$) forces the obstruction tensor to vanish, however no substantially more  general conditions on the conformal class are known to have a similar effect on $\cal O$. Our results provide such conditions. For example, we obtain:
\begin{Corollary}\label{cor0}
Under the assumptions of Theorem \ref{maintheo},    $\mathcal{O}=0$ for each of the following cases:
\begin{enumerate}
\item the conformal structure is Riemannian and $\fhol(M,c)\varsubsetneq\fso(1,n+1)$;
\item the conformal structure is Lorentzian and $\fhol(M,c)\varsubsetneq\fsu(1,n/2)$;
\item the conformal class contains an almost Einstein metric or special Einstein product (in the sense of \cite{GoverLeitner09}),
\item there is a normal conformal vector field $V$ of nonzero length or the dimension of the space of normal conformal vector fields is $\geq 2$. In particular, this is the case for Fefferman spaces over quaternionic contact structures in signature $(4k+3,4l+3)$ (characterised by $\fhol(M,c) \subset \fsp(k+1,l+1)$),
\item $(M,c)$ is spin and for $g \in c$ with spinor bundle $S^g$ there are twistor spinors $\varphi_{i=1,2} \in \Gamma(M,S^g)$ such that the spaces $\{ X \in TM \mid X \cdot \varphi_i = 0 \}$ are complementary at each point.
\end{enumerate}
\end{Corollary}

\begin{Corollary}
\label{cor1}
Under the assumptions of Theorem \ref{maintheo},  $\rk(\mathcal{O}) \leq 1$ for each of the following  cases:
\begin{enumerate}
\item $(p,q) = (3,3)$ and $\fhol(M,c)\varsubsetneq\spin(3,4)$;
\item $(p,q) = (n,n)$ and $\fhol(M,c) \subset \fgl(n+1)$;
\item $\Hol(M,c)$ fixes a nontrivial $2$-form, i.e., $(M,c)$ admits a normal conformal vector field. In particular, this applies to Fefferman conformal structures, i.e.,  to $(p,q)=(2r+1,2s+1)$ and $\fhol(M,c) \subset \fsu(r+1,s+1)$;
\item the action of $\Hol(M,c)$ on the light cone $\mathcal{N} \subset \mathbb{R}^{p+1,q+1}$ does not have an open orbit. 
\end{enumerate}
For each of these geometries one can give an explicit subspace $V \subset TM$ with  $\mathrm{Im}(\mathcal{O}) \subset V$ at each point.
\end{Corollary}

Two results in these corollaries can be found in the literature --- the statement about 
 almost Einstein \cite{fefferman/graham85} and special Einstein products \cite{GoverLeitner09} in Corollary \ref{cor0} and the statement about 
 Fefferman conformal structures  \cite{GrahamHirachi08} in Corollary \ref{cor1} ---  but the general theory as developed here allows alternative proofs of these facts.
Note also that the last two conditions in Corollary \ref{cor0} are conformally invariant and do not refer to a distinguished metric in the conformal class. 

As the  main tool in proving these results, we introduce what we call {\em the conformal holonomy distribution}. At each point $x\in M$ it is defined as 
\[\cal E_x:= \fhol_x(M,c)\cap \fg_1.\]
The vector space $\cal E_x$  can be canonically identified with a subspace in $T_x M$. When varying $x$, its dimension however may not be constant.
Instead, varying $x$ provides a stratification of the manifold into sets over which the dimension of $\cal E_x$ is constant. We will see in Theorem \ref{cotheo} that these strata are unions of the {\em curved orbits}  defined by conformal holonomy reductions, introduced recently in \cite{cgh} in the context of Cartan geometries.
Moreover we 
 will show that an open and dense set in $M$ can be covered by open sets over which the dimension of $\cal E_x $ is constant.  Very surprisingly, we find that, 
 when considered over such an open set, $\cal E$ is closely related to the aforementioned generic distributions:
 \begin{Theorem}
 Let $(M^{p,q},c)$ be a {smooth}  conformal manifold. Then there is 
 an open and dense set in $M$ that is  covered by open sets $U$ over which $\cal E|_U$ is a vector distribution. Over each such $U$,  $\cal E|_U$ is either integrable, 
 or 
 \begin{itemize}
 \item $(p,q)=(2,3)$ and    $\cal E|_U$  is a generic rank  $2$ distribution, or 
  \item $(p,q)=(3,3)$ and   $\cal E|_U$  is a generic rank  $3$ distribution.
 \end{itemize}
 In both cases, $\cal E|_U $   defines the conformal class $c$ on $U$ in the sense of \cite{nurowski04,bryant06}.
 \end{Theorem}
  We should also point out that the the statements in Theorem \ref{maintheo} remain valid when $\rk(\mathcal{O})$ at $x$ is replaced by the dimension of $\cal E_x$. 
We believe that the conformal holonomy distribution will turn out to be a powerful tool that allows to obtain not only results about the obstruction tensor but also about other aspects of special conformal structures.

This article is organised as follows: Section \ref{sec2} reviews the relevant tractor calculus and the ambient metric construction in conformal geometry. Moreover, we discuss special conformal structures that will be important in the sequel from the point of view of holonomy reductions. Section \ref{prfmain} is then devoted to the proof of the first part of Theorem \ref{maintheo}. Key ingredient is a recently established  relation between conformal and ambient holonomy \cite{CapGoverGrahamHammerl15}  . In Section \ref{holdist} we introduce the conformal holonomy distribution $\mathcal{E}$ and study its basic properties. These results are then applied in Section \ref{aplob} to derive constraints on the obstruction tensor for many families of special conformal structures, in particular those in signature $(3,3)$ discovered by Bryant.

\section{Conformal structures,  tractors and ambient metrics} \label{sec2}
\subsection{Conventions}
Let $(M,g)$ be a semi-Riemannian manifold with Levi Civita connection $\nabla^g$. 
denote by $\Lambda^k:=\Lambda^kT^*M$ the $k$-forms and by $\fso(TM)$ the endomorphism of $M$ that are skew with respect to $g$.
By $\R=\R^g\in \Lambda^2\otimes \fso(TM)$ we will denote the curvature endomorphism of $\nabla^g$, i.e. one has for all vector fields $X,Y \in \mathfrak{X}(M)$
\begin{align*}
\R^g(X,Y) = \left[\nabla^g_X, \nabla^g_Y \right] - \nabla^g_{[X,Y]}.
\end{align*}
By contraction one obtains the Ricci tensor and scalar curvature,
\begin{align*}
\Ric^g(X,Y) & := \text{tr}\left(Z \mapsto \R^g(Z,X)Y \right), \\
\scal^g & := \text{tr}_g \Ric^g,
\end{align*}
and moreover we denote by $\Sch$ the Schouten tensor
\begin{align}
\Sch := \frac{1}{n-2} \left(\Ric^g - \frac{1}{2(n-1)} \scal^g g \right).
\end{align}
Using $g$ to raise and lower indices, we will also consider $\Sch$ and $\Ric^g$ as $g-$symmetric endomorphisms of $TM$ denoted with the same symbol. The metric dual 1-form of a vector $V \in TM$ is $V^{\flat} = g(V,\cdot)$ and from a 1-form $\alpha \in T^*M$ we obtain a tangent vector via $g(\alpha^{\sharp},\cdot)= \alpha$. 
From the Schouten tensor we obtain the 
Cotton tensor $C\in \Lambda^2\otimes TM$,
\[
C^g(X,Y) := \left(\nabla^g_X \Sch \right)(Y) - \left(\nabla^g_Y \Sch \right)(X),\]
and the Weyl tensor $W \in \Lambda^2\otimes \fso(TM)$, 
\[
W^g(X,Y):= \R^g(X,Y) + X^{\flat} \otimes \Sch(Y) + \Sch(X) \otimes Y  - \Sch(Y) \otimes X - Y^{\flat} \otimes \Sch(X). \]
We will also write $C^g(Z;X,Y):=g(Z,C^g(X,Y))$ for the metric dual of $C^g$, drop the $g$ and  use the index convention $C_{kij}=C(\del_k;\del_i,\del_j)$.
\subsection{Conformal tractor calculus}
Let $(M,c)$ be a conformal manifold of signature $(p,q)$, dimension $n=p+q \geq 3$ and let $\mathcal{T} \rightarrow M$ denote the standard tractor bundle for $(M,c)$ with normal conformal Cartan connection $\nabla^{nc}$ and tractor metric $h$ as introduced in \cite{beg}. The tractor bundle $\mathcal{T}$ is equipped with a canonical filtration $\mathcal{I} \subset \mathcal{I}^{\perp} \subset \mathcal{T}$, where $\mathcal{I}$ is a distinguished lightlike line. For each metric $g \in c$ one finds distinguished lightlike tractors $s_{\pm}$ which lead to an identification 
\begin{equation}\label{idef}
\begin{array}{rcl}
\mathcal{T} &\longrightarrow &\underline{\mathbb{R}} \oplus TM \oplus \underline{\mathbb{R}}, \\
T &\longmapsto& \alpha s_- + V + \beta s_+ \ \longmapsto\  (\alpha, V, \beta)^\top 
\end{array}
\end{equation}
 under which the tractor metric becomes $h((\alpha_1,V_1,\beta_1),(\alpha_2,V_2,\beta_2)) = \alpha_1\beta_2 + \alpha_2 \beta_1 + g(V_1,V_2)$, and in this identification, $s_-$ generates $\mathcal{I}$. Under a conformal change $\widetilde{g} = e^{2 \sigma}g$, the transformation of the metric identification \eqref{idef} of a standard tractor is given by
\begin{align}
\begin{pmatrix} \alpha \\ Y \\ \beta \end{pmatrix} \longmapsto \begin{pmatrix} \widetilde{\alpha} \\ \widetilde{Y} \\ \widetilde{\beta} \end{pmatrix}= \begin{pmatrix} e^{- \sigma} (\alpha - Y(\sigma) - \frac{1}{2}\beta \cdot ||\text{grad}^g \sigma ||^2_g )\\ e^{- \sigma} (Y + \beta \cdot \text{grad}^g \sigma) \\ e^{\sigma} \beta \end{pmatrix}. \label{tra}
\end{align}
From this one observes that the image of a linear subspace $H \subset \mathcal{I}^{\perp} \subset \mathcal{T}$ under the map
\[\begin{array}{rcccl}
\mathcal{I}^{\perp} & \longrightarrow & \mathcal{I}^{\perp} / \mathcal{I} & \longrightarrow & TM, \\
\alpha s_- + V &  \longmapsto & \left[ \alpha s_- + V \right] & \longmapsto&  V
\end{array}\]
is conformally invariant, i.e., independent of the choice of $g \in c$.
 For $\nabla^{nc}$ expressed in terms of the splitting \eqref{idef} we find
 \begin{align}
 \nabla_X^{nc} \begin{pmatrix} \alpha \\ Y \\ \beta \end{pmatrix} = \begin{pmatrix} X(\alpha) - \mathsf{P}^g(X,Y) \\ \nabla^g_X Y + \alpha X + \beta \mathsf{P}^g(X) \\ X(\beta) - g(X,Y) \end{pmatrix}.\label{coder}
 \end{align}
The curvature of $\nabla^{nc}$ is given by $\R^{nc}(X,Y) = C^g(X,Y) \wedge s_-^{\flat} + W^g(X,Y)$, where we identified the bundles $\mathfrak{so}(\mathcal{T},h)$ and $\Lambda^2 \mathcal{T^*}$ by means of $h$ in the usual way by the musical isomorphisms ${}^\flat$ and ${}^\sharp$. 
{\blu }

\bigskip

Turning to adjoint tractors, it follows from identification \eqref{idef} that for fixed $g \in c$, each fiber of the bundle $\mathfrak{so}(\mathcal{T},h)$ of skew-symmetric endomorphisms of the tractor bundle can be identified with skew-symmetric matrices of the form \begin{align*}
\Phi(V,(a,A),Z) := \begin{pmatrix} -a & \mu & 0 \\ Z & A & -\mu^{\sharp}  \\ 0 & -Z^{\flat} & a \end{pmatrix}
\end{align*}
where $Z$ is a vector, $\mu$ a 1-form and $A$ is skew-symmetric for $g$.
For example, the curvature of $\nabla^{nc}$ is identified with 
\[
\R^{nc}(X,Y) = 
\begin{pmatrix} 0 & C^g(X,Y)^\flat & 0 \\ 0 & W^g(X,Y) & -C^g(X,Y)  \\ 0 & 0& 0 \end{pmatrix}
.\]
In particular, each choice of $g$ yields an obvious pointwise $|1|$ grading of $\mathfrak{so}(\mathcal{T},h)$ according to the splitting
\begin{align}
\mathfrak{g}_{-1} = \{\Phi(0,0,Z) \}, \quad \mathfrak{g}_0 = \{\Phi(0,(a,A),0) \}, \quad\mathfrak{g}_1 = \{\Phi(\mu,0,0) \},
\end{align}
with brackets given by 
\begin{equation}
\begin{aligned} \label{brackets}
[(a,A),Z] & = (a + A)Z, \\
[(a,A),\mu] &=  - \mu \circ (A+a\Id), \\
[Z,\mu] &= (\mu(Z), \mu \wedge Z^{\flat}).
\end{aligned}
\end{equation}
In particular, $[\mathfrak{g}_i,\mathfrak{g}_j] \subset \mathfrak{g}_{i+j}$. It follows that the induced derivative $\nabla^{nc}$ on a section $\Phi=\Phi(\mu,(a,A),Z)$ of $\mathfrak{so}(\mathcal{T},h)$ is given by 
\begin{align} 
\nabla^{nc}_X \Phi = \begin{pmatrix} -X(a) - \mathsf{P}^g(X,Z) - \mu(X) & \nabla^g_X \mu -\mathsf{P}^g(X,(A+a\Id) \cdot)  & 0 \\  \nabla^g_X Z-(A+a)X & \nabla^g_X A + \mu \wedge X^{\flat} - Z^{\flat} \wedge \mathsf{P}^g(X, \cdot) & -\nabla^g_X \mu^{\sharp} +(a-A) \mathsf{P}^g(X) \\ 0 &- \nabla^g_X Z^{\flat}+ (AX)^{\flat} + a X^{\flat}  & X(a) + \mathsf{P}^g(X,Z) + \mu(X)\end{pmatrix} \label{derform}
\end{align}
\subsection{Holonomy reductions of  conformal structures}\label{bspsec}
In this section we list definitions and properties of the conformal structures which have appeared in the introduction and to which the main Theorem \ref{maintheo} can be applied. They all turn out be characterised in terms of a conformal holonomy reduction. Here, for $(M^{p,q},c)$ a conformal manifold, its conformal holonomy at $x \in M$ is defined as  
\[ \Hol_x(M,c) := \Hol_x(\mathcal{T},\nabla^{nc}) \] 
and gives a class of conjugated subgroups in $\mathbf{O}(p+1,q+1)$. The interplay between {\em conformal holonomy reductions}, i.e., when $\Hol^0_x(M,c)$ is a proper subgroup of $\SO(p+1,q+1)$, and distinguished metrics in the conformal class has been the focus of active research. We will review the most important ones that are relevant in the paper.

\subsubsection{Geometries with reducible holonomy representation}
One initial result is that holonomy invariant lines $L \subset \mathbb{R}^{p+1,q+1}$ are in one-to-one correspondence to \textit{almost Einstein scales} in $c$ \cite{Gauduchon90, bailey-eastwood-gover94,Gover05, gover-nurowski04,leitner05,leistner05a} by which we mean that on an open, dense subset of $M$ there exists around each point locally an Einstein metric $g \in c$. If $L$ is lightlike, $g$ is Ricci flat and otherwise one has  $\text{sgn}(\text{scal}^g) = - \text{sgn }\langle L, L \rangle_{p+1,q+1}$. 

A holonomy-invariant nondegenerate subspace $H \subset \mathbb{R}^{p+1,q+1}$ of dimension $k \geq 2$ corresponds locally and off a singular set to the existence of a \textit{special Einstein product} in the conformal class \cite{armstrong07conf,leitner04,ArmstrongLeitner12}. Here, we say that a pseudo-Riemannian manifold $(M,g)$ is a special Einstein product if $(M,g)$ is isometric to a product $(M_1,g_1) \times (M_2,g_2)$, where $(M_i,g_i)$ are Einstein manifolds of dimensions $k-1$ and $n-k-1$ for $k \geq 2$ and in case $k \neq 2,n$ we additionally require that
\begin{align*}
\text{scal}^{g_1} = - \frac{(k-1)(n-2)}{(n-k+1)(n-k)} \text{scal}^{g_2} \neq 0.
\end{align*}

Finally, if actually $H \subset \mathbb{R}^{p+1,q+1}$ is totally degenerate, of dimension $k+1\geq 2$ and holonomy invariant, there exists --- again locally and off a singular set --- a metric $g \in c$ admitting a $\nabla^g$-invariant and totally degenerate distribution $\mathcal{L} \subset TM$ of rank $k$ which additionally satisfies  $\mathrm{Im}(\text{Ric}^g) \subset \mathcal{L}$, as has been shown in \cite{leistner05a,leistner-nurowski12,Lischewski15}.

\subsubsection{Geometries defined via normal conformal Killing forms}
Suppose next that $\Hol(M,c)$ lies in the isotropy subgroup of a $(k+1)$-form, i.e. there exists a $\nabla^{nc}$-parallel tractor $k+1$-form $\widehat{\alpha} \in \Gamma(M,\Lambda^{k+1}\mathcal{T}^*)$. Such holonomy reductions have been studied in \cite{leitner05}.
For fixed $g \in c$, consider the splitting of $\mathcal{T}$ wrt. $g$ and write $\widehat{\alpha}$ as
\begin{align}
\widehat{\alpha} = s_+^{\flat} \wedge \alpha + \alpha_0 + s_-^{\flat} \wedge s_+^{\flat} \wedge \alpha_{\pm} + s_-^{\flat} \wedge \alpha_- \label{splitncform}
\end{align}
for uniquely determined differential forms $\alpha, \alpha_{0},\alpha_{\pm}, \alpha_-$ on $M$. 
The $k-$form $\alpha \in \Omega^k(M)$ turns out to be \textit{normal conformal (nc)}, that is $\alpha$ is a conformal Killing form subject to additional conformally covariant differential normalisation conditions that can be found in \cite{leitner05}. Moreover, $\alpha_{0},\alpha_{\pm}, \alpha_-$ can be expressed in terms of $\alpha$ and $\nabla^g$. Conversely, every normal conformal Killing form determines a parallel tractor form. The situation simplifies considerably if $k=1$, i.e. there is a parallel adjoint tractor. In this case it is convenient to consider the metric dual $V = \alpha^{\sharp} \in \mathfrak{X}(M)$ of the associated normal conformal Killing form $\alpha$, which is a \textit{normal conformal vector field}. By this, we mean that $V$ is a conformal vector field which additionally satisfies $C^g(V,\cdot) = W^g(V,\cdot) = 0$. \\
Examples for manifolds admitting normal conformal vector fields are so called \textit{Fefferman spaces} \cite{Fefferman76}. They yield conformal structures $(M,c)$ of signature $(2r+1,2s+1)$ defined on the total spaces of $S^1$-bundles over strictly pseudoconvex CR manifolds. From the holonomy point of view they are (at least locally) equivalently characterised by the existence of a parallel adjoint tractor \cite{ leitnerhabil,CapGover10}, which is an almost complex structure for the tractor metric, i.e. $\Hol(M,c) \subset \mathbf{SU}(r+1,s+1)$. Here, we used a result from \cite{leitner08,CapGover10} which asserts that unitary conformal holonomy is automatically special unitary. 

Other geometries that are characterised by the existence of distinguished normal conformal vector fields include pseudo-Riemannian manifolds $(M,g)$ of signature $(4r+3, 4m+ 3)$ with conformal holonomy group in the symplectic group
$\Sp(r+1,m + 1) \subset \SO(4r+4, 4m+ 4)$, see \cite{alt08}. The models of such manifolds are $S^3$-bundles over a quaternionic contact manifold equipped with a canonical conformal structure, introduced in \cite{Biquard00}.

\subsubsection{Conformal holonomy and twistor spinors}
If $(M,c)$ is actually spin for one, and hence all, $g \in c$, the presence of conformal Killing spinors always leads to reductions of $\Hol(M,c)$. To formulate these, let $S^g \rightarrow M$ denote the real or complex spinor bundle over $M$ which possesses a spinor covariant derivative $\nabla^{S^g}$ and vectors act on spinors by Clifford multiplication $cl = \cdot$, see \cite{ba81}. Given these data, the spin Dirac operator is given as $D^g = cl \circ \nabla^{S^g}$.  Now assume that $(M,g)$ admits a twistor spinor,  i.e. a section $\varphi \in \Gamma(M,S^g)$ solving 
\begin{align}
\nabla^{S^g}_X \varphi + \frac{1}{n} X \cdot D^g \varphi = 0. \label{teq}
\end{align}
Equation \eqref{teq} is conformally invariant, see \cite{bfkg}, and to $\varphi$ we can associate the union of subspaces 
\[\cal{\cal  L}_{\varphi}:= \{ X \in TM \mid X \cdot \varphi = 0\} \subset TM, \] 
which does not depend on the choice of $g \in c$. Equation \eqref{teq} can be prolonged, see \cite{bfkg}, and using this prolonged system it becomes immediately clear that a twistor spinor $\varphi$ is equivalently described as parallel section $\psi$ of the spin tractor bundle associated to $(M,c)$. Its construction can be found in \cite{leitnerhabil}, for instance. As $\psi$ is parallel, it is at each point annihilated by $\mathfrak{hol}_x(M,c)$ under Clifford multiplication, i.e.
\begin{align}
\mathfrak{hol}_x(M,c) \cdot \psi_x = 0, \quad\forall x \in M. \label{holts}
\end{align}

\subsubsection{Exceptional cases}
Finally we describe conformal structures in dimension 5 and 6 with holonomy algebra contained in $\fg_2 \subset \mathfrak{so}(3,4)$, the non-compact simple Lie algebra of dimension $14$, or in $\spin(3,4) \subset \fso(4,4)$, respectively. They turn out to be closely related to generic distributions: 

Recall that a distribution $\cal D$ of rank $2$  on a $5$-manifold $M$   is {\em generic} if $\left[\cal D,[\cal D,\cal D]\right]+ \left[\cal D,\cal D\right]+\cal D=TM$. It is known by work of Nurowski \cite{nurowski04}  that  $\cal D$ canonically defines a conformal structure $c_{\cal D}$ of signature $(2,3)$ on $M^5$ whose conformal holonomy is reduced to $\mathfrak{g}_2\subset \fso(3,4)$, see also \cite{cap-sagerschnig09}. Analogously, 
a distribution $\cal D$ of rank $3$  on a $6$-manifold $M$   is {\em generic}
if  $\left[\cal D,\cal D\right]+\cal D=TM$, and Bryant showed in 
 \cite{bryant06}
  that  $\cal D$ canonically defines a conformal structure $c_{\cal D}$ of signature $(3,3)$ on $M$ whose conformal holonomy is reduced to  $\mathfrak{spin}(4,3)\subset \fso(4,4)$.
In both cases, the holonomy characterisation implies  that $(M,c_{\cal D})$ admits a parallel tractor $3$- or $4$-form, respectively. Moreover, \cite{HammerlSagerschnig11} shows that there is in both cases a distinguished twistor spinor $\varphi$ which encodes $\mathcal{D}$ in the sense that 
\begin{align}
\cal{\cal  L}_{\varphi} = \mathcal{D}
\end{align}   
at each point.

\subsection{Conformal ambient metrics} \label{ambsec}
Let $(M,c)$ be a conformal manifold of dimension $\geq 3$. For our purposes we do not need the general theory of ambient metrics as presented in \cite{fefferman-graham07},  which can be consulted for more details, but it suffices to deal with ambient metrics which are in normal form w.r.t.~some $g \in c$. A (straight) pre-ambient metric in normal form w.r.t.~$g \in c$ is a pseudo-Riemannian metric $\widetilde{g}$ on an open neighborhood $\widetilde{M}$ of $\{1\} \times M \times \{0 \} $ in $\mathbb{R}^+ \times M \times \mathbb{R}$ such that  for $(t,x,\rho) \in \widetilde{M}$ we have
\begin{align}
\widetilde{g} = 2t dtd\rho + 2\rho dt^2 + t^2 g_{\rho}(x) \label{nform}
\end{align}
with $g_{0}= g$. We call $(\widetilde{M},\widetilde{g})$ an {\em ambient metric for $(M,[g])$ in normal form} with respect to $g$ if 
\begin{itemize}
\item 
$\widetilde{\Ric} = O(\rho^{\infty})$ if $n$ is odd, and 
\item 
$\widetilde{\Ric}= O(\rho^{\frac{n}{2}-1})$ and  $\mathrm{tr}_g \left(\rho^{1-\frac{n}{2}} \widetilde{\Ric}_{|TM \otimes TM}\right) = 0$ along $\rho=0$, if $n$ is even.
\end{itemize}
The existence and uniqueness assertion for ambient metrics \cite{fefferman/graham85,fefferman-graham07} states that for each choice of $g$ there is an ambient metric in normal form w.r.t.~$g$. In all dimensions $n\geq 3$,  $g_{\rho}$ has an expansion of the form $g_{\rho} = \sum_{k \geq 0} g^{(k)} \rho^k$ starting with
\[ g_{\rho} = g + 2\rho\, \mathsf{P}^g + O(\rho^2), \]
and in odd dimensions the Ricci flatness condition determines $g^{(k)}$ for all $k$, whereas in even dimensions only the  $g^{\left(k < \frac{n}{2} \right)}$ and the trace of $g^{\left(\frac{n}{2}\right)}$ are determined.

We shall sometimes work with ambient indices $I \in \{0,i,\infty \}$, where $i$ are indices for coordinates on $M$, $0$ refers to $\partial_t$ and $\infty$ to $\partial_{\rho}$, i.e., $T\widetilde M \ni V=
V^0\del_t+V^i\del_i+V^\infty\del_\rho$.
For the Levi-Civita connection 
of any metric of the form  \eqref{nform} one computes \cite[Lemma 3.2]{fefferman-graham07},
 \begin{equation}
\label{ambientconnection}
\begin{array}{rcl}
\widetilde{\nabla}_{\del_i}\del_j&=& -\frac{1}{2}\, t\,\dot g_{ij} \del_t +\Gamma^k_{ij}\del_k+(\rho\dot g_{ij}-g_{ij})\del_\rho,
\qquad \widetilde{\nabla}_{\del_t}\del_t\ =\  
\widetilde{\nabla}_{\del_\rho}\del_\rho\ =\ 0,
\\[2mm]
\widetilde{\nabla}_{\del_i}\del_t&=& \frac{1}{t}\del_i,
\qquad 
\widetilde{\nabla}_{\del_i}\del_\rho\ =\  \frac{1}{2}g^{kl}\dot g_{il}\del_k,
\qquad \widetilde{\nabla}_{\del_\rho}\del_t\ =\  \frac{1}{t}\del_\rho,
\end{array}
\end{equation}
where, abusing  notation, $g_{ij}$ denotes the components of $g_\rho$ and $\Gamma^k_{ij}$ the Christoffel symbols of $g_\rho$.
In particular, $T := t \partial_t$ is an Euler vector field for $(\widetilde{M},\widetilde{g})$, i.e.
\begin{align}
\widetilde{\nabla} T = Id.
\end{align}

For $n$ even a conformally invariant $(0,2)$-tensor on $M$, the {\em ambient obstruction tensor} $\mathcal{O}$, obstructs the existence of smooth solutions to $\widetilde{\Ric} = O(\rho^{\frac{n}{2}})$. For $\widetilde{g}$ in normal form w.r.t.~$g$ it is given by
\begin{align}
\mathcal{O} = c_n \left(\rho^{1-\frac{n}{2}} (\widetilde{\Ric}_{|TM \otimes TM})\right)_{\rho = 0}, \label{defobstr}
\end{align}
where  $c_n$ is some known nonzero constant, cf. \cite{fefferman-graham07}. From this  can deduce that $\mathcal{O}$ is trace- and divergence free.

Tractor data can be recovered from ambient data as shown in \cite{cap-gover03}. For ambient metrics in normal form w.r.t.~$g \in c$, this reduces to the following observation, see \cite{graham-willse11} for more details:  Identify $M$ with the level set $\{ \rho = 0, t=1 \}$ in $\widetilde{M}$. Obviously, $T\widetilde{M}_{|M}$ splits into $\mathbb{R}\partial_t \oplus TM \oplus \mathbb{R}\partial_{\rho}$, which is isomorphic to the $g-$metric identification of the tractor bundle $\mathcal{T}$ under the map
\begin{align}
\partial_t \longmapsto s_-, \quad TM \stackrel{Id}{\longmapsto} TM, \quad \partial_{\rho} \longmapsto s_+. \label{identific}
\end{align}
The map \eqref{identific} is an isometry of bundles over $M$ w.r.t.~$\widetilde{g}$ and $h$ and the pullback of $\widetilde{\nabla}$, the Levi Civita connection of $\widetilde{g}$, to  $T\widetilde{M}_{|M}$  coincides with \eqref{coder}. This also follows directly from an inspection of \eqref{tra} and \eqref{ambientconnection}. With  these identifications, for fixed $g \in c$  we view the tractor data as restrictions of ambient data for an ambient metric which is in normal form with respect to  $g$.

\section{The ambient obstruction tensor and conformal holonomy} \label{prfmain}
We outline how the obstruction tensor can be identified with a distinguished subspace of the infinitesimal conformal holonomy algebra at each point. This requires some preparation:

Let $V$ be a vector space. The standard action $\#$ of  $End(V)$ on $V$ extends to an action on the space $T^{r,s}V$ of $(r,s)$ tensors over $V$. This action will be denoted by the same symbol. Thus, $End(V) \otimes End(V)$ acts on $T^{r,s}V$ with a double $\#-$action, explicitly given by
\begin{align}
(A \otimes B) \# \# (\eta) = A \# (B \# \eta). \label{hash}
\end{align}
Given a pseudo-Riemannian manifold $(N,h)$, we can view its curvature tensor $\mathrm{R}^h$ as section of the bundle $\mathfrak{so}(N,h) \otimes \mathfrak{so}(N,h)$, and applying \eqref{hash} pointwise yields an action $\mathrm{R}^h \# \#$ of the curvature on arbitrary tensor bundles of $N$.

Returning to the original setting, let $(M,g)$ be a pseudo-Riemannian manifold of even dimension and let $(\widetilde{M},\widetilde{g})$ be an associated ambient metric which is in normal form w.r.t.~$g$. Let $\widetilde{\Delta} = \widetilde{\nabla}_A \widetilde{\nabla}^A$ denote the usual connection Laplacian on the ambient manifold. In \cite{gope}  a modified Laplace type operator 
\begin{align}
\slashed{\Delta} = \widetilde{\Delta} + \frac{1}{2} \aR \# \# 
\end{align}
is introduced and will be used in the subsequent calculations.


The previous observations enable us to prove the main result of this section:
\begin{Theorem} \label{obstrhol}
Let $(M,c=[g])$ be of even dimension $>2$. For every $g \in c$ one has \[s_-^{\flat} \wedge \mathcal{O}^{\flat}(X) \in \mathfrak{hol}_x(M,[g])\] for all $x \in M$ and $X \in T_xM$.
\end{Theorem}

\bprf
The proof uses the notion of {\em infinitesimal holonomy}:
within in the Lie algebra $\mathfrak{hol}_x(M,c)$ of $\Hol_x(M,c)$ at a point $x \in M$, we consider the {\em infinitesimal holonomy algebra at $x$}, i.e. the Lie algebra of iterated  derivatives of the tractor curvature evaluated at $x$,
\[
\hol^\prime_x(M,c)
:=\text{span}_\rr
\{\nabla^{nc}_{X_1}\left( \ldots \left( \nabla^{nc}_{X_{l-1}}\left( \R^{nc}(X_{l-1},X_l)\right)\right)\right)(x)\mid l\ge 2, X_1,\ldots, X_l\in \mathfrak{X}(M)
\}.\]
For more details on the infinitesimal holonomy refer to \cite[Chap. II.10]{ko-no1} or \cite{Nijenhuis53,Nijenhuis54}. We will in fact show that $
s_-^{\flat} \wedge \mathcal{O}^{\flat}(X) \in \mathfrak{hol}^\prime_x(M,[g])$ for all $x \in M$ and $X \in T_xM$.

Assume first that $n>4$. Let $(\widetilde{M},\widetilde{g})$ be an associated ambient manifold for $(M,[g])$ which is in normal form w.r.t.~some fixed $g$ in the conformal class. 
For $x \in M$ let
\[ \mathfrak{hol}_x(\widetilde{M},\widetilde{g}) := \text{span}_{\mathbb{R}}\left\{\widetilde{\nabla}_{X_1}(...\widetilde{\nabla}_{X_{l-2}} (\aR(X_{l-1},X_l)))(x) \mid l \geq 2, X_i \in \mathfrak{X}(\widetilde{M}) \right\} \]
denote the infinitesimal holonomy algebra of $(\widetilde{M},\widetilde{g})$ at $x$ and for $k\geq 0$ let $\mathfrak{hol}^k_x(\widetilde{M},\widetilde{g})$ denote the subspace of elements for which at most $k$ of the $X_i$ have a not identically zero $\partial_{\rho}$-component. Then    \cite[Theorem~3.1]{CapGoverGrahamHammerl15}
 asserts that under the identifications from Section \ref{ambsec} we have 
\begin{align}
{\mathfrak{hol}^\prime_x(M,c)} = \mathfrak{hol}_x^{\frac{n}{2}-2}(\widetilde{M},\widetilde{g}). \label{holeq}
\end{align}
Indeed, for $(\widetilde{M},\widetilde{g})$ which is in normal form w.r.t.~$g$, equality \eqref{holeq} can be verified as follows:

From the identifications from Section \ref{ambsec} one obtains immediately the inclusion $\subset$ in \eqref{holeq}. In order to prove the converse, we obtain with \cite[Lemma 3.1]{graham-willse11} and \cite[Propositon 6.1]{fefferman-graham07} that
\begin{equation}
\begin{aligned} \label{indbegin}
\aR(\partial_i,\partial_j)(x) &= \R^{nc}(\partial_i,\partial_j)(x), \\
\aR(\partial_t,\partial_{I})(x) &= 0, \\
\aR(\partial_{\rho},\partial_{i})(x) &= 3 g^{kl} \left(\nabla^{nc}_{\partial_k} \R^{nc}\right)(\partial_l, \partial_i)(x).
\end{aligned}
\end{equation}
 The right sides of these expressions clearly lie in {$\mathfrak{hol}^\prime_x(M,c)$}. To proceed, using linearity and commuting covariant derivatives, it suffices to prove that 
\begin{align}
\left( \widetilde{\nabla}_{X_i}^k \widetilde{\nabla}_{\partial_{\rho}}^l \widetilde{\nabla}_{\partial_t}^j \aR \right) (Y,Z) (x) \in {
\mathfrak{hol}^\prime_x(M,c)}, \label{genericel}
\end{align}
where $k,j,l \geq 0$, $X_i \in T_xM$, $Y,Z \in T_x \widetilde{M}$ and $l \leq \frac{n}{2}-1$ or $\frac{n}{2}-2$ (depending on whether one of $Y,Z$ has a $ \partial_{\rho}$-component): Given an element of the form \eqref{genericel} one first applies Proposition 6.1 from \cite{fefferman-graham07}, which rewrites $\partial_t$ derivatives of $\widetilde{\R}$, and obtains a linear combination of elements of the form \eqref{genericel} with $j = 0$ and $Y,Z$ have no $\partial_t$-component. Thus, it suffices to prove \eqref{genericel} for $j=0$. This is then achieved by induction over $l$. Indeed, for $l=0$ the statement follows from \eqref{indbegin}. Furthermore, we may assume that $Y= \partial_{\rho}$ (otherwise all differentiations are tangent to $M$ or we use the second Bianchi identity) and $Z \in T_xM$. However, applying Lemma 3.1 from \cite{graham-willse11}, which rewrites $\aR  (\partial_{\rho},Z)$ in terms of derivatives of $\widetilde{\R}$ tangent to $M$, and then applying the second Bianchi identity and the induction hypothesis, shows the claim for an element of the form \eqref{genericel}. This proves \eqref{holeq}.

Using again the identifications from Section \ref{ambsec}, we will now show that for $x \in M$ and $X \in T_xM$ we have 
\begin{align}
 \partial_{t}^{\flat}(x) \wedge \mathcal{O}(X)^{\flat}(x) \in \mathfrak{hol}_x^{\frac{n}{2}-2}(\widetilde{M},\widetilde{g}). \label{equal}
 \end{align} 
{With this,  equality (\ref{holeq}) and the inclusion $\mathfrak{hol}^\prime_x(M,c)\subset \mathfrak{hol}(M,c)$ will imply the statement of Theorem~\ref{obstrhol}.}
In order to verify property (\ref{equal}), note that, as observed in \cite{gope},
on any pseudo-Riemannian manifold one has (in abstract indices)
\begin{align}
4 \widetilde{\nabla}_{A_1} \widetilde{\nabla}_{B_1} \widetilde{\Ric}_{A_2 B_2} =  \slashed{\Delta} \aR_{A_1 A_2 B_1 B_2} - \widetilde{\Ric}_{CA_1} {\aR^C}_{~A_2 B_1 B_2} + \widetilde{\Ric}_{CB_1} {\aR^C}_{~B_2 A_1 A_2} , \label{divric0}
\end{align}
here $A_1,A_2$ and $B_1,B_2$ are pairwise skew-symmetrised. Indeed, \eqref{divric0} is a straightforward consequence of the second Bianchi identity. As in our situation $\widetilde{\Ric} = O(\rho^{\frac{n}{2}-1})$ it follows that 
\begin{align}
4 \widetilde{\nabla}_{A_1} \widetilde{\nabla}_{B_1} \widetilde{\Ric}_{A_2 B_2} =  \slashed{\Delta} \aR_{A_1 A_2 B_1 B_2} + O(\rho^{\frac{n}{2}-1}). \label{divric}
\end{align}
Via induction and using that $\widetilde{\nabla}_{\partial_{\rho}} \partial{\rho} = 0$ one derives from \eqref{divric}
\begin{align}
\left(\widetilde{\nabla}_{\partial_{\rho}}^{\frac{n}2-3} \slashed{\Delta} \aR\right)(Y_1,Y_2, Z_1,Z_2)(x) = 4\left(\widetilde{\nabla}_{\partial_{\rho}}^{\frac{n}2-3} \widetilde{\nabla}_{Y_1} \widetilde{\nabla}_{Z_1}\widetilde{\Ric}\right)(Y_2,Z_2)(x), \label{diffric}
\end{align}
where now $x \in M$, $Y_i,Z_i$ are ambient vector fields and $Y_1,Y_2$ as well as $Z_1,Z_2$ are skew-symmetrised.
Now let $Y_1 = \partial_{\rho}$ and $Y_2 = X \in \mathfrak{X}(M)$ and insert $\widetilde{\Ric} = O(\rho^{\frac{n}{2}-1})$  into the right side of \eqref{diffric}. It follows that the resulting expression is zero unless one of $Z_i$ is proportional to $\partial_{\rho}$ and the other one is a tangent vector $Y \in T_xM$. For this choice of vectors we have
\begin{align}
\left(\widetilde{\nabla}_{\partial_{\rho}}^{\frac{n}2-3} \slashed{\Delta} \aR\right)(\partial_{\rho},X, \del_\rho, Y )(x) = 
\left(\widetilde{\nabla}_{\partial_{\rho}}^{\frac{n}2-1}\widetilde{\Ric}\right)(X,Y)(x)
, \label{obstkr0}
\end{align}
for $X,Y\in TM$.
Hence, 
 by definition \eqref{defobstr}
 one obtains a multiple of $\mathcal{O}(X,Y)$,
\begin{align}
\left(\widetilde{\nabla}_{\partial_{\rho}}^{\frac{n}2-3} \slashed{\Delta} \aR\right)(\partial_{\rho},X)(x) = 
k(n) \cdot \partial_{t}^{\flat}(x) \wedge \mathcal{O}(X)^{\flat}(x), \label{obstkr}
\end{align}
for some nonzero numerical constant $k(n)$ which depends only on the dimension $n$. To proceed, we analyse the left side in \eqref{obstkr}. Equations \eqref{ambientconnection} show that the ambient Laplacian applied to some tensor field $\eta$ has an expansion of the form
\begin{equation}\label{delform}
\widetilde{\Delta} \eta \ =\  \widetilde{g}^{IJ} \widetilde{\nabla}_I \widetilde{\nabla}_J \eta\\
 \ =\  \frac{1}{t} \widetilde{\nabla}_{\partial_{\rho}} (\widetilde{\nabla}_{\partial_t} \eta ) + \frac{1}{t} \widetilde{\nabla}_{\partial_{t}} (\widetilde{\nabla}_{\partial_{\rho}} \eta ) - \frac{{ 2}\rho}{t^2} \widetilde{\nabla}_{\partial_{\rho}} (\widetilde{\nabla}_{\partial_{\rho}} \eta) + f \, \widetilde{\nabla}_{\partial_{\rho}} \eta + \widetilde{D} \eta,
\end{equation}
where $f$ is a certain known function on $\widetilde{M}$ and $\widetilde{D}$ is an operator of the form
 \begin{equation}\label{dform}
\widetilde{D} \eta = \sum_{i,j} a_{ij} \widetilde{\nabla}_i (\widetilde{\nabla_j}\eta) + 
\sum_{K \in \{k,0\}}b_{K} \widetilde{\nabla}_K \eta .\end{equation}
We conclude inductively that for an arbitrary ambient tensor field $\eta$ and an element $Z \in \mathfrak{so}(T\widetilde{M})$  one has
\begin{equation}
\begin{aligned} \label{lapl}
\eta = O(\rho^l) &\Rightarrow \widetilde{\Delta}^k \eta = O(\rho^{l-k}), \\
Z(x) \in \mathfrak{hol}_x^l(\widetilde{M},\widetilde{g})  &\Rightarrow  (\widetilde{\Delta}^k Z)(x) \in \mathfrak{hol}_x^{k+l}(\widetilde{M},\widetilde{g}).
\end{aligned}
\end{equation}
Moreover, a straightforward linear algebra calculation using the algebraic Bianchi identity for the ambient curvature reveals that (in abstract ambient indices and with  brackets denoting skew symmetrisation)
\[
(\aR \# \# \aR)_{ABCD} 
=
2 \, \aR_{ABVW}\aR^{VW}_{\ \ \ \ CD}
+8\, 
\aR_{C~[A~}^{\ \ P \ \ Q}\, 
\aR_{B]PDQ}^{~\ }
-2\left(
\widetilde{\Ric}^{V}_{~[A} \aR_{B]VCD}
+
\widetilde{\Ric}^{V}_{~[C} \aR_{D]VAB}
\right).
\]
For each $A,B$ the first term on the right hand side is contained in the holonomy algebra as it is a linear combination of curvature tensors. Similarly, the second term is a linear combination of commutators of curvature tensors and hence also in the holonomy algebra.
Differentiating this $\frac{n}{2}-3$ times in $\partial_{\rho}$ direction and using that $\widetilde{\Ric}$ vanishes to order $\frac{n}{2}-1$ shows via induction that
\begin{align*}
\left(\widetilde{\nabla}_{\partial_{\rho}}^{\frac{n}{2}-3}(\aR \# \# \aR)\right)(\partial_{\rho},X)(x) \in \mathfrak{hol}^{\frac{n}{2}-2}_x(\widetilde{M},\widetilde{g}).
\end{align*}
Next, we focus on the $\rho$-derivatives of $\widetilde{\Delta}$ in \eqref{obstkr}. Using the form of $\widetilde{\Delta} $ in \eqref{delform} and \eqref{dform} and calculating mod $\mathfrak{hol}^{\frac{n}{2}-2}_x(\widetilde{M},\widetilde{g})$, they are given by
\[
\left(\widetilde{\nabla}_{\partial_{\rho}}^{\frac{n}2-3}{\widetilde{\Delta}} \aR\right)(\partial_{\rho},X)(x) \   =\   \widetilde{\nabla}_{\partial_{\rho}}^{\frac{n}2-3}\left(\widetilde{\Delta} \aR(\partial_{\rho},X)\right)(x) \\
\  =\  l(n)  \left(\widetilde{\nabla}_{\partial_{\rho}}^{\frac{n}2-2}\aR\right)(\partial_{\rho},X)(x)
\]
for some numerical constant $l(n)$. Thus, we have found that for $x \in M$, $X \in T_xM$
\begin{align}
 k(n) \partial_{t}^{\flat}(x) \wedge \mathcal{O}(X)^{\flat}(x) - l(n) \cdot \left(\widetilde{\nabla}_{\partial_{\rho}}^{\frac{n}{2}-2}\aR\right)(\partial_{\rho},X) = E_X \label{enned}
 \end{align}
for some $E_X \in \mathfrak{hol}_x^{\frac{n}{2}-2}(\widetilde{M},\widetilde{g})$. Now insert $Y \in T_xM$ and $\partial_{\rho}$ into the 2-forms in \eqref{enned}. One obtains
\begin{align}
k(n) \mathcal{O}(X,Y) - l(n) \, \left(\widetilde{\nabla}_{\partial_{\rho}}^{\frac{n}{2}-2}\aR\right)(\partial_{\rho},X, Y , \partial_{\rho}) = E_X(Y,\partial{\rho}). \label{ergebnis}
\end{align}
By   \cite[Proposition 6.6]{fefferman-graham07} we have
\begin{align*}
2\left(\widetilde{\nabla}_{\partial_{\rho}}^{\frac{n}{2}-2}\aR\right)(\partial_{\rho},\partial_i, \partial_j , \partial_{\rho}) = \text{tf}(\partial_{\rho}^{\frac{n}{2}} g_{ij}) + K_{ij},
\end{align*}
where $K_{ij}$ can be expressed algebraically in terms of $(\partial_{\rho}^{k}g_{ij})_{|\rho = 0}$, $k < \frac{n}{2}$, as well as $g^{ij}_{|\rho = 0}$. Moreover, as follows from reviewing the above argument, $E$ can be expressed algebraically in terms of derivatives of $g_{\rho}$ and its inverse in $M$-directions and at most $\frac{n}{2}-1$ derivatives in $\rho$-direction and $\mathcal{O}$ is a natural tensor invariant. But then, as  the ambiguity, i.e. the term $\text{tf}(\partial_{\rho}^{\frac{n}{2}} g_{ij})$, can be arbitrary, equation \eqref{ergebnis} can only be true if $l(n) = 0$ from which the theorem follows if $n > 4$.

In general, it holds in every dimension that for $X \in TM$ one has \[\text{tr}_g \nabla^{nc}_{\cdot}\R^{nc}(X,\cdot) =   (n-4) C(X;\cdot, \cdot) + B(X) \wedge s_-^{\flat} \in \mathfrak{hol}(M,c),\] 
where 
\[ B_{ij}=\nabla^k C_{ijk}-P^{kl}W_{kijl},
\]
is the Bach tensor, and where for each pair $j,k$ we understand $\R^{nc}_{jk}$ as an element in $\Lambda^2\cal T^*$.
From this observation the theorem follows in case $n=4$, as here $\mathcal{O}$ is a multiple of the Bach tensor
 \eprf
\bbem 
Consider the case $n=6$. It is an entirely mechanical process to turn the formulas in \cite{gope}, section 4B into an explicit formula for derivatives of the tractor curvature which gives a more explicit proof of Theorem \ref{obstrhol} for this dimension.
In order to make this more explicit, assume that there is a metric $g \in c$ and a totally lightlike subspace $\cal L \subset TM$ such that $\mathrm{Im}(\Ric^g) \subset \cal L$ and $\cal L$ is $\nabla^g$ invariant. Such geometries correspond to invariant null subspaces which are invariant under $\Hol(M,c)$ and are of importance in Section \ref{spin43}.
 Let $\nabla$ denote the tractor derivative $\nabla^{nc}$ coupled to $\nabla^g$. One can explicitly compute for this case that
\[
g^{ij} s_-^{\flat} \wedge \mathcal{O}_{mi} \partial_j^{\flat} =   g^{ij} g^{kl} \nabla_i \nabla_j \nabla_k \R^{nc}_{ml} + 4 \mathsf{P}^{ij} \nabla_i \R^{nc}_{mj}+ 2 \left[\R^{nc}_{mi}, \nabla^{nc}_j {\R^{nc}}^{ij} \right] + 2 {C_{m}}^{kl} \R^{nc}_{kl}.
\]
\ebem 
\section{The conformal holonomy distribution}
In this section we will introduce and study the fundamental object that provides us with the link between conformal holonomy and the ambient obstruction tensor.

\subsection{The conformal holonomy distribution} \label{holdist}
Let $(M,c=[g])$ be a conformal manifold of arbitrary signature $(p,q)$ and dimension $n=p+q$. For $x \in M$ consider the conformal holonomy algebra $\mathfrak{hol}_x(M,c) \subset \mathfrak{so}(\mathcal{T}_xM,h_x)$. Fix $g \in c$. Theorem \ref{obstrhol} motivates us to study the following subspaces of $T_xM$,
\begin{align}
\mathcal{E}^g_x := \{ pr_{T_xM} \text{Im}(A) \mid A \in \mathfrak{hol}_x(M,c), A \mathcal{I} = 0, h(A \mathcal{I}^{\bot}, \mathcal{I}^{\bot}) = 0 \} \subset T_xM.
\end{align}
It follows immediately from the transformation formulas that $\mathcal{E}^g_x$ does not depend on the choice of $g \in c$, so that we can  write $\mathcal{E}_x$. With respect to  $g \in c$, however, $\mathcal{E}_x$ is identified with the space of elements of the holonomy algebra that are of the form $s_-^{\flat} \wedge X^{\flat}$ for some $X \in T_xM$. Equivalently and more invariantly, the space $\mathcal{E}_x$ can be identified with the space $\mathfrak{hol}_x(M,c) \cap \mathfrak{g}_1$. 
We call the subset of $TM$ defined by 
\[\cal E :=\bigcup_{x\in M} {\cal E}_x \subset TM\] 
the {\em conformal holonomy distribution}. This is a slight abuse of terminology, as the dimension of $\cal E_x$ may vary with $x$, so that $\cal E$ is not a vector distribution in the usual sense. 
Indeed, the holonomy algebras w.r.t.~to different base points are related by the adjoint action of elements in $\mathbf{O}(p+1,q+1)$ that generically do not lie in the stabiliser of $s_-$. 
Instead, 
 define a function on $M$ by \[ r^{\cal E} (x):= \text{dim }\mathcal{E}_x. \] 
The function $r^{\cal E}$ 
need not be constant over $M$ but leads to an obvious stratification
\[ M = \bigcup_{k = 0}^n M_k, \] 
where $M_k = \{ x \in M \mid r^{\cal E}(x) = k \}$. 

\subsection{Relation to the curved orbit decomposition}\label{cosec}
We will now proceed to establish a relation
between the stratification defined by $\mathcal{E}$ and the curved orbit decomposition  for holonomy reductions of arbitrary Cartan geometries in \cite{cgh}. When doing this, we will restrict to the case that $\mathfrak{hol}(M,c)$ equals the stabiliser of some tensor:

Starting with the tractor data $(\mathcal{T} \rightarrow M,h,\nabla^{nc})$ one recovers an underlying Cartan geometry as follows \cite{cap-gover03}: Fix a lightlike line $L \subset \mathbb{R}^{p+1,q+1}$ and at each point $x \in M$ consider the set of all linear, orthogonal maps $\mathbb{R}^{p+1,q+1} \rightarrow \mathcal{T}_x$ which additionally map $L$ to $\mathcal{I}_x$. This defines a principal $P$- bundle $\mathcal{G} \rightarrow M$, where $P \subset G = \mathbf{O}(p+1,q+1)$ is the stabiliser subgroup of $L$. Then the tractor connection $\nabla^{nc}$  induces a Cartan connection $\omega \in \Omega^1(\mathcal{G},\mathfrak{g})$ of type $(G,P)$, i.e., $\omega$ is equivariant w.r.t.~the $P$-right action, reproduces the generators of fundamental vector fields, and provides a global parallelism $T \mathcal{G} \cong \mathcal{G} \times \mathfrak{g}$. In this way, $(\mathcal{G} \rightarrow M, \omega)$ is a Cartan geometry of type $(G,P)$. Conversely, one obtains the standard tractor bundle from these data as $\mathcal{T} = \mathcal{G} \times_P \mathbb{R}^{p+1,q+1} = \widehat{\mathcal{G}} \times_G \mathbb{R}^{p+1,q+1}$, where $\widehat{\mathcal{G}} = \mathcal{G} \times_P G$ denotes the enlarged $G$-bundle. The Cartan connection $\omega$ lifts to a principal bundle connection $\widehat{\omega}$ on $\widehat{\mathcal{G}}$ and $\nabla^{nc}$ is then the induced covariant derivative on the associated bundle $\mathcal{T}$.

Now assume that there is a faithful representation $\rho$ of $G$ on some vector space $V$  with associated vector  bundle $\mathcal{H} =  \widehat{\mathcal{G}} \times_G V$ and induced covariant derivative $\nabla^{\mathcal{H}}$  such that $\Hol(M,c)$ equals pointwise the stabiliser of a $\nabla^{\mathcal{H}}$-parallel section $\psi \in \Gamma(M,\mathcal{H})$ (if actually $(M,c)$ is spin, the same discussion is possible for spin coverings of the groups and bundles under consideration). Such a $\psi$ is equivalently encoded in a $G$-equivariant map $s : \widehat{\mathcal{G}} \rightarrow V$ which is constant along $\widehat{\omega}$-horizontal curves. To this situation the general machinery developed in \cite{cgh} applies and one defines for $x \in M$ the $P$- type of $x$ (w.r.t.~$\psi$) to be the $P$-orbit $s(\mathcal{G}_x) \subset V$. Then $M$ decomposes into a union of initial submanifolds $M_{\alpha}$ of elements with the same $P$-type, where $\alpha$ runs over all possible $P$-types, which in turn can be found by looking at the homogeneous model $G \rightarrow G/P$. In   \cite{cgh} the $M_{\alpha}$ are called {\em curved orbits} and it was shown that they  carry a naturally induced Cartan geometry  of type $(H,P\cap H)$.

\begin{Theorem}\label{cotheo}
If $\Hol(M,c)$ equals to the stabiliser of a tensor, then the subsets of $M$ on which $r^\mathcal{E}$ is constant  are unions of curved orbits in the sense of \cite{cgh}. In particular, they are unions of initial submanifolds.
\end{Theorem}
\bprf
We fix a  curved orbit $M_{\alpha}$ with element $x_1$. By definition, $x_2 \in M_{\alpha}$ if and only if
\begin{align}
s(\mathcal{G}_{x_1}) = s(\mathcal{G}_{x_2}). \label{ptype}
\end{align} 
We unwind the condition \eqref{ptype} as follows: Let $u_{x_i} \in \mathcal{G}_{x_i}$ and let $[u_{x_i}] : V \ni v \mapsto [u_{x_i},v] \in \mathcal{H}_{x_i}$ denote the associated fibre isomorphism. As $\rho$ is faithful the holonomy group $Hol_{u_{x_i}}(\widehat{\omega}) \subset G$ will coincide with the stabiliser of $[u_{x_i}]^{-1}\psi_{x_i} \subset V$ under the $(\rho,G)$-action. Moreover \eqref{ptype} is equivalent to the existence of $p \in P$ such that 
\begin{align*}
\rho(p) \left([u_{x_1}]^{-1} \psi_{x_1} \right) = [u_{x_2}]^{-1} \psi_{x_2},
\end{align*}
from which one deduces that 
\begin{align}
Ad(p^{-1})(\mathfrak{hol}_{u_{x_1}}(\widehat{\omega})) = \mathfrak{hol}_{u_{x_2}}(\widehat{\omega}). \label{holre}
\end{align}
Using that $[\mathfrak{g}_i,\mathfrak{g}_j] \subset \mathfrak{g}_{i+j}$, one sees that \eqref{holre} restricts to a map between the $\mathfrak{g}_1$-components of $\mathfrak{hol}_{u_{x_i}}(\widehat{\omega})$ which therefore have the same dimension. As $\mathfrak{hol}_x = [u_x] \circ \mathfrak{hol}_{u_x}(\widehat{\omega}) \circ [u_x]^{-1} \subset \mathfrak{so}(\mathcal{T}_x,h_x)$ and $[u_x]$ preserves the lightlike line by definition of $\mathcal{G}$, we obtain that also the dimensions of $\mathfrak{hol}_{x_i} \cap \mathfrak{g}_1$ agree. Consequently, $r^{\cal E}$ is constant  on the curved orbit $M_{\alpha}$.\eprf

Theorem \ref{cotheo} shows that, in general,   the holonomy distribution $\mathcal{E}$ as studied here  will induce a stratification of $M$ that  is {\em coarser} than the curved orbit decomposition in \cite{cgh}. The following example shows that in some cases it  induces the same stratification.
 \begin{bsp}
Suppose $(M,c)$ is of Riemannian signature and $\Hol_x(M,c)$ equals the stabiliser of some tractor $\zeta_x \in \mathcal{T}_x$. For any metric $g \in c$ write $\zeta = ( \alpha , Y , \beta )^\top$ for smooth functions $\alpha, \beta$ and a vector field $Y$ on $M$. Evaluating $\nabla^{nc} \zeta = 0$ using \eqref{coder} yields 
\[Y = \text{grad}^g \beta,\quad \alpha g = \beta \mathsf{P}^g - \text{Hess}^g (\beta).\] An element $V^{\flat} \wedge s_-^{\flat}$ lies in $\mathfrak{hol}_x(M,c) \cap \mathfrak{g}_1$ if and only if $d\beta(V) = 0$ as well as $\beta \cdot V = 0$ at $x$. If $h(\zeta,\zeta) \neq 0$, we conclude that 
\[M = M_0 \cup M_{n-1}, \quad\text {with $ M_0 = \{ \beta \neq 0 \}$ and $M_{n-1} = \{ \beta = 0 \}$.}\] For $x \in M_{n-1}$ we have $\mathcal{E}_x = \text{ker }d\beta \neq T_xM$. In particular, $M_{n-1}$ is a smooth embedded submanifold of $M$. Similarly, if $h(\zeta,\zeta) = 0$, we have \[M = M_{0} \cup M_{n} = \{\beta \neq 0 \} \cup \{ \beta = 0\}.\] Here $ \{ \beta = 0\}$  consists only of isolated points because  $\beta(x) = 0$ implies  that $d\beta(x) = 0$ and $\text{Hess}^g(\beta)(x)$ is proportional to $g_x$.
\end{bsp}

\subsection{Open sets adapted to the holonomy distribution}\label{adaptsec}
We analyse the function $r^{\cal E}$ in more detail. Obviously, if $\mathfrak{hol}_x(M,c)$ is generic at some  point of $M$, i.e., if $\mathfrak{hol}_x(M,c)=\so(p+1,q+1)$, then $r^{\cal E}\equiv n$. Conversely, one finds:

\begin{Proposition} \label{lightlike}
Suppose that there is a curve $\gamma$ in $M$ with $g(\dot{\gamma},\dot{\gamma}) \neq 0$ and $r^{\cal E} \circ \gamma \equiv n$. Then $\mathfrak{hol}(M,c)$ is generic. In particular, $r^{\cal E} \equiv n$. 
\end{Proposition}
\bprf
All calculations are carried out w.r.t.~some fixed $g \in c$. By assumption, $s_-^{\flat} \wedge V^{\flat} \in \mathfrak{hol}_{\gamma(t)}(M,c)$ for every vector field $V$ along $\gamma$. Applying $\nabla^{nc}_{\dot{\gamma}}$ to this expression using \eqref{derform} reveals that 
\begin{align} - g(V,\dot{\gamma}) s_-^{\flat} \wedge s_+^{\flat} + \dot{\gamma}^{\flat} \wedge V^{\flat} \in \mathfrak{hol}_{\gamma(t)}(M,c). \label{hole}\end{align}
Letting $V = \dot{\gamma}$ shows that $s_-^{\flat} \wedge s_+^{\flat} \in \mathfrak{hol}_{\gamma(t)}(M,c)$. Moreover, letting $(V_1,V_2,\dot{\gamma})$ being mutually orthogonal to each other and taking the Lie brackets of the expressions \eqref{hole} with $V= V_1$ and $V = V_2$, respectively, shows that \[||\dot{\gamma}||^2 V^{\flat}_1 \wedge V^{\flat}_2 \in \mathfrak{hol}_{\gamma(t)}(M,c).\] But this establishes that $\mathfrak{g}_0 \in \mathfrak{hol}_{\gamma(t)}(M,c)$. Thus, $\mathfrak{g_1} \oplus \mathfrak{g}_0 \in \mathfrak{hol}_{\gamma(t)}(M,c)$. Differentiating elements $\dot{\gamma}^{\flat} \wedge V^{\flat} \in \mathfrak{hol}_{\gamma(t)}(M,c)$ in direction of $\gamma$, where $V$ is again a vector field along $\gamma$ shows using \eqref{derform} that also $\mathfrak{g}_{-1} \cap \dot{\gamma}^{\perp}$ is contained in the inifitesimal holonomy along $\gamma$ and differentiating $s_-^{\flat} \wedge s_+^{\flat}$ along $\gamma$ shows that all of $\mathfrak{g}_{-1}$ is contained in the holonomy. Thus, $\mathfrak{hol}_{\gamma(t)}(M,c)$ is generic along $\gamma$, and thus generic everywhere.
\eprf

In order to continue with our analysis, we need to show that there are \textit{sufficiently many} open sets $U$ on which $r^{\cal E}  $ is constant, i.e., such that $\cal E|_{U}$ is a vector bundle,  and on which there is a basis of local smooth sections of $U \rightarrow \mathcal{E}$. For this purpose we define:
An open set $U \subset M$ is an  {\em $\mathcal{E}$-adapted open set} if 
\begin{enumerate}
\item $r^{\cal E} \equiv  k$  constant  on $U$,
\item there are smooth and pointwise linearly independent sections $V_1,...V_k : U  \rightarrow \mathcal{E} $.
\end{enumerate}
Then it holds:
\begin{Theorem} \label{eadapated}
For each open set 
$U \subset M$  there exists an $\cal E$-adapted open subset $V \subset U$. In particular, 
there is an open dense subset of $M$ which is the union of $\cal E$-adapted open sets.
\end{Theorem}

\bprf
After restricting $U$ if necessary, we may assume that $U$ is contained in a coordinate neighborhood for $M$. It is then possible to choose a local basis of $\mathfrak{hol}_x(M,c)$ over $U$ which depends {smoothly}  on $x$. Write such a basis as  
\begin{align} U \ni x \mapsto (v^{\flat}_i(x) \wedge s^{\flat}_-  + A^i(x)), \label{localb}
\end{align}
where $i=1,...,m:=\text{dim }\mathfrak{hol}(M,c)$, for certain $v_i \in T_xM$ and $A^i \in \mathfrak{g_0} \oplus \mathfrak{g}_{-1}$. With respect to the fixed coordinates we may think of the $A^i =( A^i_{jk})_{j,k}$ as $\mathfrak{so}(p+1,q+1)$-matrices. Let \[\widetilde{A}^i := (A^i_{11}, A^i_{12}, \ldots , A^i_{n+1, n+2} ,A^i_{n+2, n+2})^\top\] and introduce the $m\times (n+2)^2$-matrix $A:= \begin{pmatrix} \widetilde{A}^1 & ... & \widetilde{A}^m \end{pmatrix}$. Elementary linear algebra reveals that
\begin{align}
r^{\cal E}(x) = k \Leftrightarrow k = \text{dim ker }A = \text{dim }\mathfrak{hol}_x(M,c) - \text{rk } A_x. \label{dimconst}
\end{align}
The set of matrices with rank greater or equal to some fixed integer is open in the set of all matrices. Thus, it follows from the equivalence \eqref{dimconst} that $\{ x \mid r^{\cal E}(x) \leq k \}$ is open in $M$. In particular, $(r^{\cal E})^{-1}(0) = \{x \mid r^{\cal E}(x) \leq 0 \}$ is open and $r^{\cal E} < n$ is an open condition.

Assume now that there is $x \in U$ with $r^{\cal E}(x) = 0$. It follows that $r^{\cal E}=0$ on some open subset $V \subset U$. Thus the claim follows for this case. Otherwise, we have $ r^{\cal E}\geq 1$ everywhere. If there is $x \in U$ with $r^{\cal E}(x) = 1$, it follows that there is an open neighbourhood $V$ in $U$ with $r^{\cal E} \leq 1$ of $x$ in $U$. Thus, $r^{\cal E}=1$ on $V$. Otherwise we have  $r^{\cal E} \geq 2$ on $U$ etc. So the statement regarding the existence of $V$ with $r^{\cal E}|_{V} =: l =  \text{constant}$ follows inductively. The above proof starts with a smooth local basis \eqref{localb} and constructs (on an open subset of $V$) via smooth linear algebra operations a basis on $V$ of the form $(\widetilde{v}^{\flat}_{i=1,...,l} \wedge s_-^{\flat},...)$. It is thus clear that the $\widetilde{v}_i$ depend smoothly on $x \in V$ and yield local sections.

Finally, if every open set in $M$ contains an $\cal E$-adapted open subset, 
 the union of all $\mathcal{E}$-adapted open sets is open and dense in $M$. 
\eprf
By virtue of this theorem, after restricting to an open and dense subset of $M$ if necessary, we may from now on always assume that $M$ is the union of $\mathcal{E}$-adapted open sets. In particular, the level sets of $r^{\cal E}$ are then (possibly empty) unions of $\mathcal{E}$-adapted open sets. From this point of view, we may restrict ourselves to such open sets in the following local analysis. Note that restricting to open and dense subset in context of Cartan holonomy reductions is a basic feature of the curved orbit decomposition as revealed in \cite{cgh}.

\begin{Proposition}
 Let $U \subset M$ be a $\mathcal{E}$-adapted open set. Then $\mathcal{E}_x$ is a totally lightlike subspace of $T_xM$ for every $x \in U$ or $\mathfrak{hol}(M,c)$ is generic. \label{furr}
\end{Proposition}
\bprf
Let $V$ be a vector field defined on $U$ such that $s_-^{\flat} \wedge V^{\flat}(x) \in \mathfrak{hol}_x(M,c)$ for $x \in U$. Differentiating in direction of some $X \in TM$ using \eqref{derform} reveals  that
\begin{align} -\nabla^{nc}_X(s_-^{\flat} \wedge V^{\flat})(x)= g(V,X) s_-^{\flat} \wedge s_+^{\flat} + X^{\flat} \wedge V^{\flat}  + (\nabla_X V)^{\flat} \wedge s_-^{\flat} \in \mathfrak{hol}_{x}(M,c). \label{hole2}
\end{align}
Suppose that there is $x \in U$ with $g(V,V)(x) \neq 0$. It follows that $g(V,V) \neq 0$ on some open neighborhood $x \in W \subset U$. Let $X$ be orthogonal to $V$ on $W$. As $\mathfrak{hol}_{x}(M,c)$ is a Lie algebra with the usual commutator Lie bracket, it follows that on $W$ also
\begin{align}
\left[  X^{\flat} \wedge V^{\flat}  + (\nabla_X V)^{\flat} \wedge s_-^{\flat},s_-^{\flat} \wedge V  \right] = -g(V,V) X^{\flat} \wedge s_-^{\flat} \in \mathfrak{hol}(M,c) 
\end{align}
Thus, $r^\cal E|_{W} = n$ and the statement follows from Proposition \ref{lightlike}.
\eprf
%
%
%
%

\subsection{Rank and integrability of the holonomy distribution}
Interestingly, it turns out that, at least locally, $\mathcal{E}$ always integrable or it is maximally non integrable and one of the exceptional holonomy reductions occurs. More precisely, we will see that if $\cal E$ is not integrable, $M$ is of dimension $5$ or $6$, $\cal E$ is generic and of rank $2$ or $3$, respectively, and  $\mathfrak{hol}(M,c)$ is $\mathfrak{g}_2$ or $\mathfrak{spin}(4,3)$, respectively.\

In order to analyse the integrability of $\cal E$, we need some preparations.
\begin{Proposition} \label{openorbit}
Let $(M^n,c)$ be a conformal manifold of even dimension. Either there is an open dense subset of $M$ on which  $r^\cal E\le 1$  or $\Hol^0(M,c)$ acts on the lightcone $\mathcal{N} \subset \mathbb{R}^{p+1,q+1}$ with an open orbit.
\end{Proposition}

\bprf
Suppose first that $r^{\cal E} \geq 2$ on some open set $U \subset M$. After restricting to an open, dense subset of $U$, if necessary, we may assume that ${U}$ is an $\mathcal{E}$-adapted open set. We may also assume that the holonomy is not generic and hence that $\cal E$ is lightlike. Let $V$ be a local section of $\mathcal{E}$ and let $V'$ be a lightlike vector field with $g(V,V') = 1$. Moreover, let $X \in (V,V')^{\perp}$. We have on $U$
\begin{align}
\nabla^{nc}_{V'} (s^{\flat}_- \wedge V^{\flat})  &= s_-^{\flat} \wedge s_+^{\flat} + A_1\  \in\  \mathfrak{hol}(M,c),\label{oo1} \\
\nabla^{nc}_X \left(\nabla^{nc}_{V'} (s^{\flat}_- \wedge V^{\flat}) \right) &= - X^{\flat} \wedge s_+^{\flat} + A_2\  \in\  \mathfrak{hol}(M,c)
\label{oo2}
\end{align}
where $A_{1,2} \in \mathfrak{g}_0 \oplus \mathfrak{g}_1=\mathfrak{p}$. As $r^{\cal E} \geq 2$ on $U$ and $\mathcal{E}$ is totally lightlike, linear algebra shows that  at $x \in U$, equation (\ref{oo2}) implies
\begin{align}
\mathfrak{so}(p+1,q+1) = \mathfrak{hol}_x(M,c) + \mathfrak{p}.
\end{align}
This, together with equation \eqref{oo1} shows  that the orbit of $\Hol^0(M,c)$ through $s_-\in \cal N$ has dimension $n+1$, i.e., it is open.
Otherwise, the subset of $M$ on which $ r^{\cal E} \leq 1$ is dense. It is also open as follows from the proof of Theorem \ref{eadapated}.  
\eprf
In relation to this proposition, we point out that conformal structures for which the holonomy group acts not only with an open orbit on $\cal N$, but transitively and irreducibly on the homogeneous model were classified in \cite{Alt12}.

\begin{Proposition} \label{ncform}
Suppose that $(M,[g])$ admits a nc-Killing form $\alpha \in \Omega^k(M)$. Then $V^{\flat} \wedge \alpha = 0$ for every $V \in \mathcal{E}$. 
\end{Proposition}

\bprf
Following the discussion in Section \ref{bspsec}, every nc-Killing $k$-form $\alpha$ uniquely determines a parallel tractor $(k+1)$-form $\widehat{\alpha}$. With respect to  a metric $g$ in the conformal class, decompose $\widehat{\alpha}$ as in \eqref{splitncform}. Pointwise, $\widehat{\alpha}$ is  annihilated by the action $\#$ of $\mathfrak{hol}(M,c)$ on forms. In particular, one has for every $V \in \mathcal{E}_x$ that
\[ (s_-^{\flat} \wedge V^{\flat}) \# \widehat{\alpha}_x = 0 \]
Inserting \eqref{splitncform}, one immediately obtains that $V^{\flat} \wedge \alpha$ = 0.
\eprf

\begin{Proposition} \label{splitnd}
Suppose $M$ is orientable and the action of $\fhol(M,c)$ leaves invariant a nontrivial non degenerate subspace of $\mathbb{R}^{p+1,q+1}$. Then $\mathcal{E} = 0$ on an open, dense subset of $M$. 
\end{Proposition}

\bprf
As the holonomy invariant space (of dimension $k+1$) is nondegenerate and $M$ is orientable, there is actually a decomposable parallel tractor form in $\Omega^{k}\cal T^*$. The associated nc-Killing form $\alpha$ is of the form $\alpha = t_1 \wedge...\wedge t_{k}$ defining  a $k-$dimensional nondegenerate subspace $H \subset TM$ on an open, dense subset of $M$ as follows from the discussion in \cite{leitner05},
Thus, Proposition \ref{ncform} implies that $\mathcal{E} \subset H$ on an open dense subset $M'$ of $M$. On the other hand, by Proposition \ref{lightlike}, $\mathcal{E}$ is over $M'$ contained in a totally degenerate subspace. We conclude $\mathcal{E}_{|M'} = 0$. 
\eprf

\begin{Proposition}\label{splitndlight}
Suppose that $\Hol(M,c)$ fixes a totally lightlike (w.r.t.~$h$) subbundle $\mathcal{H} \subset \mathcal{T}$. 
Then there is on an open and dense subset of $M$ at least locally a metric $g\in c$ such that w.r.t.~$g$ 
\begin{align}
\mathcal{H} = \mathbb{R}s_+ \oplus \cal L,  \label{grep}
\end{align} 
with $\cal L\subset TM $ a  $\nabla^g$-parallel distribution containing $\cal E$ and  the  image of 
 $\Ric^g$.
\end{Proposition}
\bprf
The existence of a parallel distribution $\cal L\subset TM$ containing the image of $\Ric^g$ was proven in \cite{Lischewski15}. To see that at each $x\in M$, the fibre $\cal L_x$ contains $\cal E_x$, consider $V\in \cal E_x$ such that $s_-^{\flat} \wedge V^{\flat} \in \fhol_x(M,c)$. Then $(s_-^{\flat} \wedge V^{\flat} )(s_+)=V$ lies in $ \cal H$ which shows that $\cal E\subset \cal L$.
\eprf

These results enable us to prove the main result of this section:
\begin{Theorem} \label{eint}
Let $U \subset M$ be a $\mathcal{E}$-adapted open set. Then  exactly one of the following cases occurs on $U$:
\begin{enumerate}
\item $\mathcal{E}$ is integrable.
\item The dimension of $M$ is $5$ and 
 $\mathcal{E}$ is a generic rank $2$ distribution. Moreover, 
 $\fhol(M,c)=\fg_2$ and hence 
the conformal structure 
$c=c_{\cal E}$ is  defined by the generic distribution $\cal E$.
\item The dimension of $M$ is $6$ and  $\mathcal{E}$ is a generic rank $3$ distribution. Moreover,  
$\fhol(M,c)=\spin(3,4)$ and  
the conformal structure
 $c=c_{\cal E}$ is  defined by the generic distribution $\cal E$.
\end{enumerate}
\end{Theorem}

\bprf
If $\mathfrak{hol}(M,c)$ is generic the statement is trivial as $\mathcal{E} = TM$ in this case. Thus, we may assume that the holonomy algebra is reduced and by the previous Proposition, $\mathcal{E}_x$ is a totally lightlike subspace of $T_xM$ for $x \in U$. 

Let $V_1, V_2$ be vector fields on $U$ such that $s_-^{\flat} \wedge V^{\flat}_{i=1,2} \in \mathfrak{hol}_x(M,c)$ for $x \in U$. 
It follows that 
\begin{align}
\nabla^{nc}_{V_1} (s_-^{\flat} \wedge V^{\flat}_2) - \nabla^{nc}_{V_2} (s_-^{\flat} \wedge V^{\flat}_1) = - V^{\flat}_1 \wedge V^{\flat}_2 + s_-^{\flat} \wedge ([V_1,V_2])^{\flat} \in \mathfrak{hol}(M,c). \label{commhol}
\end{align}
Moreover, let $X$ be a vector field on $U$ which is orthogonal to $V_i$ for $i=1,2$. It follows from evaluating $[\nabla^{nc}_{X} (s_-^{\flat} \wedge V^{\flat}_1) , \nabla^{nc}_{X} (s_-^{\flat} \wedge V^{\flat}_2)]$ that 
\begin{align}
 2 g(\nabla_X V_1,V_2) X^{\flat} \wedge s_-^{\flat} + g(X,X) V^{\flat}_1 \wedge V^{\flat}_2 \in \mathfrak{hol}(M,c). \label{bracket}
\end{align}
Combining \eqref{commhol} and \eqref{bracket} it follows for $X$ orthogonal to $(V_1,V_2)$ that
\begin{align}
X \cdot g(\nabla_X V_1,V_2) & \in \mathcal{E} \text{ for }g(X,X) = 0, \label{eq1}\\
[V_1,V_2] - \frac{2g(\nabla_X V_1,V_2)}{g(X,X)} \cdot X & \in \mathcal{E} \text{ for }g(X,X) \neq 0. \label{eq2}
\end{align}
Now we distinguish several cases: Obviously the statement is trivial in case $r^{\cal E} \leq  1$. Thus, we may assume that $V_1,V_2$ are linearly independent. Fix a local $g-$pseudo-orthonormal basis $(s_1,....,s_n)$ over $U$ such that 
\begin{align}
\mathcal{E} = \text{span}(V_i := s_{2i-1} + s_{2i} \mid i = 1,...r^{\cal E} ) \label{adbasis}
\end{align}
Moreover, let $V_i' := s_{2i-1} - s_{2i}$ for $i = 1,...e$. That is, $g(V_i,V_j') = 2\delta_{ij}$.

\bigskip

\textit{Case 1:} $r^{\cal E} \geq 3$ and $n > 6$.
In \eqref{eq1} let $X = V_3'$. It follows that $g(\nabla_{s_5} V_1,V_2) = g(\nabla_{s_6} V_1,V_2)$. But then letting $X = s_5,s_6$, \eqref{eq2} can only be true if $[V_1,V_2] - f \cdot V_3' \in \mathcal{E}$ for some function $f$. On the other hand, applying \eqref{eq2} to $X = s_n$ reveals that $[V_1,V_2] - h \cdot s_n \in \mathcal{E}$ for some function $h$. But this can only be true if $f=h = 0$, i.e. $[V_1,V_2] \in \mathcal{E}$. 

\bigskip

\textit{Case 2:} $r^{\cal E}= 2$ and $n > 5$.
In complete analogy to the previous case, we obtain that $[V_1,V_2] - f s_5 \in \mathcal{E}$ for some function $f$ as well as $[V_1,V_2] - h s_6 \in \mathcal{E}$ for some function $h$ from which one has to conclude that $f=h=0$.

\bigskip

\textit{Case 3:} $r^{\cal E}=2$ and $n=4$. 
$M$ is necessarily of signature $(2,2)$. It follows from \eqref{eq1} that for $i,j,k \in \{1,2\}$ we have $g(\nabla_{V_i},V_j,V_k) = 0$. But this implies that $g(\nabla_{V_1} V_2 - \nabla_{V_2} V_1,V_k) = 0$, i.e. $[V_1,V_2] \in \mathcal{E}^{\perp} = \mathcal{E}$.

\bigskip

It remains to show that in signatures $(3,2)$ and $(3,3)$ and $\mathcal{E}$ being of dimension $2$ resp.~$3$ and not integrable, $\mathcal{E}$ is generic. 

First, let us  consider signature $(3,2)$ and assume that $\mathcal{E}$ is not integrable. In particular, $\mathcal{E}$ is of rank $2$ on an open and dense set. One could  proceed with the proof for this case analogously as with the  $(3,3)$ case below.  However, as we are considering a conformal structure in {\em odd} dimension, one of the main results in \cite{CapGoverGrahamHammerl15} yields that $\mathfrak{hol}(M,c)$ is the holonomy algebra of a Ricci flat pseudo-Riemannian manifold of signature $(4,3)$. If the standard action of $\mathfrak{hol}(M,c)$ was reducible, then by Propositions \ref{splitnd} and \ref{splitndlight},
 $\mathcal{E}$  would either be either zero or contained in an integrable  totally lightlike distribution, both contradicting the assumptions in the current case.
Thus, the action of the holonomy algebra is irreducible and from $\mathcal{E} \neq TM$ and the pseudo-Riemannian version of the Berger list it follows that $\mathfrak{hol}(M,c) = \mathfrak{g}_{2}$, where $\mathfrak{g}_{2}$ denotes the non-compact simple Lie algebra of dimension $14$. For this case, however $\mathcal{E}$ is generic.  
This follows from the discussion of $\fg_2$-conformal structures in Section \ref{bspsec} in complete analogy to the proof of Corollary \ref{bryantcor2} in Section \ref{spin43}.

Let us now treat the 6-dimensional case. Fix a local basis $(V_1,V_2,V_3,V_1',V_2',V_3')$ for $TM$ over $U$ as specified in \eqref{adbasis}
such that $g(V_i,V_j') = 2\delta_{ij}$. 
Moreover, without loss of generality, we may  assume that 
\begin{align}
[V_1,V_2] \notin \mathcal{E}. \label{notint}
\end{align}
 From \eqref{eq1} we obtain $g(\nabla_{V_3'}V_1,V_2) = 0$ and \eqref{eq2} applied to $X = V_3 + V_3'$ then yields 
 \begin{align}
 [V_1,V_2] - g(\nabla_{V_3} V_1,V_2) V_3' \in \mathcal{E}. \label{derv1}
 \end{align}
 We conclude from \eqref{notint} that $g(\nabla_{V_3} V_1,V_2) \neq 0$. Moreover, it follows from subtracting $\nabla_{V_2} (s_-^{\flat} \wedge V_1^{\flat}) \in \mathfrak{hol}(M,c)$ from \eqref{commhol} that 
 \begin{align}
 \nabla_{V_2} V_1 + [V_1,V_2] \in \mathcal{E}. \label{ine}
\end{align}
In complete analogy to the derivation of \eqref{derv1} we obtain that $[V_1,V_3] -g(\nabla_{V_2} V_1,V_3) V_2' \in \mathcal{E}$. Inserting \eqref{ine} and then using \eqref{notint} and \eqref{derv1} reveals that the coefficient $g(\nabla_{V_2} V_1,V_3)$ is nonzero. The same argument applies to $[V_2,V_3]$ and we conclude that there are nowhere vanishing functions $f_{k}$ for $k=1,2,3$ such that
\[ [V_i, V_j] = \epsilon_{ijk} f_k V'_k \text{ mod } \cal E. \]
In particular, $[\mathcal{E},\mathcal{E}] = TM$.

It remains to show that in this case we have $\fhol(M,c)=\spin(3,4)$.
Using \eqref{derform}, it is straightforward to compute that the 15 elements, $i,j=1,2,3$, 
\begin{equation}
\label{15elements}
\begin{array}{rl}
s_-^{\flat} \wedge V^{\flat}_i,& \\
\nabla^{nc}_{V'_i} (s_-^{\flat} \wedge V^{\flat}_j),& \\
\nabla^{nc}_{V_i} (s_-^{\flat} \wedge V^{\flat}_j),& i < j 
\end{array}
\end{equation}
are pointwise linearly independent in $\mathfrak{hol}(M,c) \cap \mathfrak{p}$. 
Then Proposition \ref{openorbit} comes into play, which ensures that $\mathfrak{so}(p+1,q+1) = \mathfrak{hol}(M,c) + \mathfrak{p}$ and hence that  dim $\mathfrak{hol}(M,c) \geq 15+6 = 21$, which is the dimension of $\mathfrak{spin}(4,3)$. Then the equality $\mathfrak{hol}(M,c)=\mathfrak{spin}(4,3)$, and with it the last point in the  theorem,  follows from Lemma \ref{2128lemma} below.
\eprf
\begin{Lemma}\label{2128lemma}
Let $\fh\varsubsetneq \fso(4,4)$ be irreducible of dimension at least $21$. Then $\fh=\fspin(3,4)$. 
\end{Lemma}
\bprf
Since  $\fh$ acts irreducibly, it is reductive. Then either $\fh$   is  semisimple and the complexified representation $\mathbb{C}\otimes \R^{4,4}$ is irreducible, or 
  $\fh \subset \mathfrak{u}(2,2)$ and $\mathbb{C}\otimes \R^{4,4}$ is not irreducible 
(see for example  \cite[Section 2]{Di-ScalaLeistner11}). The second case however is excluded by the assumption $\dim(\fh)\ge 21$. Hence, we may consider $\fh^\C\subset \fso(8,\C)$ semisimple acting irreducibly on $\C^8$. 
Inspecting the dimensions of simple complex Lie algebras below $28$, it turns out that the only possibilities for $\fh$, apart from $\fso(7,\C)$, are  $\fsl_5\C$ and $\fsl_2\C\oplus \fsl_3\C$.  Then  $\fsl_5\C$ is excluded as it does not have an irreducible representation of dimension $8$. On the other hand, any irreducible representation of $\fsl_2\C\oplus \fsl_3\C$ is  a tensor product of irreducible representations, which is excluded as $\fsl_3\C$ does not have an  irreducible representations of dimension $2$ or $4$.
\eprf


Finally, we want  to derive universal integrability conditions for Weyl and Cotton tensor for conformal manifolds with reduced holonomy.

\begin{Proposition}
Let $(M,c)$ be a conformal manifold with non generic holonomy. Locally, and off a singular set there is a totally degenerate subspace $\mathcal{L} \subset TM$, which is integrable if $(p,q) \notin \{(3,2),(3,3) \}$, such that 
\begin{align}
W(\mathcal{L},\mathcal{L}^{\perp}) & = 0, \label{int1}\\
(n-4)C(\mathcal{L},\mathcal{L}^{\perp}) & = 0. \label{int2}
\end{align}
In even dimensions, one has  $\mathrm{Im}(\mathcal{O}) \subset \mathcal{L}$. In particular, if a conformal manifold in even dimension $ \geq 4$ admits a parallel tractor (of any type) other than the tractor metric, then the conformally invariant system \eqref{int1} - \eqref{int2} either becomes a nontrivial integrability condition on the curvature (and it couples $\mathcal{O}$ to the curvature) or $\mathcal{O} = 0$.
\end{Proposition}

\bprf
We restrict the local analysis to $\mathcal{E}-$adapted open sets and let $\mathcal{L} = \mathcal{E}$. The conditions \eqref{int1} and \eqref{int2} are easily seen to be an equivalent reformulation of
\begin{align}
\left[\R^{nc}(X,Y),s_-^{\flat} \wedge V^{\flat} \right] & \in \mathfrak{hol}(M,c), \\
\left[\text{tr}_g \nabla{.} \R^{nc}(\cdot,X),s_-^{\flat} \wedge V^{\flat} \right] & \in \mathfrak{hol}(M,c),
\end{align}
where $X, Y \in TM$ and $V \in \mathcal{E}$. The statement follows from the definition of $\mathcal{E}$ and Theorems \ref{obstrhol} and \ref{eint}.
\eprf




\section{Applications to the obstruction tensor} \label{aplob}
Recall that according to Theorem \ref{obstrhol} the image of the obstruction tensor $\cal O$  is contained in the holonomy distribution $\cal E$. In this section we 
 apply the results about $\cal E$ to obtain the results in Corollaries \ref{cor0} and \ref{cor1}. In the following we will always assume we have given a smooth conformal manifold $(M,c)$ of even dimension $n$ and with obstruction tensor $\cal O$. We view $\mathcal{O}$  as a $(1,1)$-tensor by means of some $g \in c$ and define the rank of $\mathcal{O}$ at a point to be the rank of this $(1,1)$-tensor. The holonomy reductions we will  consider now  were described in Section \ref{bspsec}.

\subsection{The obstruction tensor and holonomy reductions} We start with a well-known case of a conformal holonomy reduction, the case of a parallel standard tractor. 
The existence of  parallel standard tractor is equivalent to the existence of an open dense subset in $M$, on which the conformal class contains  local  Einstein metrics.
It is well known 
since   \cite[Proposition 3.5]{fefferman/graham85}, see also \cite[Theorem 4.3]{gope} and \cite{fefferman-graham07}) that the existence of local Einstein metrics in the conformal class forces $\cal O=0$. Our Theorem \ref{obstrhol} provides us with an independent and alternative proof:
\begin{Corollary}\label{einsteincor}
 If locally on an open and dense subset of $M$  there is an Einstein metric $g \in c$, then $\cal O=0$.
 \end{Corollary}
 \bprf
 Given an Einstein  metric on $U \subset M$ and splitting the tractor bundle over $U$ w.r.t.~$g$, there is on $U$ a parallel standard tractor $T = -\frac{\text{scal}^g}{2n(n-1)} \cdot s_- + s_+$. In particular, $\mathfrak{hol}_x(U,[g])  T_x = 0$. Theorem~\ref{obstrhol} yields $(s_-^{\flat} \wedge \mathcal{O}(X)^{\flat})(T) = \cal O(X) = 0$ on $U$ for each $X \in TU$ which is equivalent to $\mathcal{O} = 0$ on $U$.
\eprf

A weaker condition than admitting a parallel tractor is the existence of a subspace that is invariant under the conformal holonomy. In this situation Propositions \ref{splitnd} and \ref{splitndlight} imply:
\begin{Corollary}\label{splitnd-cor}
Suppose $M$ is orientable and the action of $\Hol(M,c)$ leaves invariant a nontrivial  subspace   $\cal H$ of $\mathbb{R}^{p+1,q+1}$. Then we have the following alternative (possibly replacing $\cal H$ with $\cal H\cap\cal H^\perp$ if it is degenerate):
\begin{enumerate}
\item If $\cal H$ is non-degenerate, then $\cal O=0$.
\item If $\cal H$ is totally lightlike, 
then, locally on an open dense subset of $M$  there is  a metric $g\in c$
and 
a  $\nabla^g$-parallel distribution
$\cal L\subset TM $  containing  the image of 
 $\Ric^g$ and of $\cal O$.
 \end{enumerate}
\end{Corollary}
Specialising the total lightlike case in this corollary further, in  Section \ref{spin43} we will consider  Bryant's conformal structures as examples.  Another example is the following:
\begin{bsp}
 Suppose that $M$ is of split signature $(n,n)$ and that $\Hol(M,c)$ leaves invariant two complementary totally lightlike distributions $\mathcal{H} \oplus \mathcal{H}' = \mathcal{T}$, i.e. $\Hol(M,c) \subset \mathbf{GL}(n+1,\rr) \subset \mathbf{SO}(n+1,n+1)$. Such conformal structures arise from Fefferman type constructions starting with $n-$dimensional projective structures, see \cite{HammerlSagerschnig11twistor,HammerlSagerschnigSilhanTaghavi-ChabertZadnik15}. For $\mathcal{H}$ and $\mathcal{H}'$ define $L$ and $L'$ as above and fix a local metric $g$ such that $\mathcal{H}$ is of the form \eqref{grep} on some set $U \subset M$. Elementary linear algebra shows that on $U$ the space $L \cap L'$ is at each point at most $1$-dimensional. Moreover, we have from the conformal covariance of $\mathcal{O}$ and Corollary \ref{splitnd-cor} that  $\mathrm{Im}(\mathcal{O}) \subset L \cap L'$. It follows that the rank of $\mathcal{O}$ is less or equal to one on an open, dense subset of $M$. 
\end{bsp}

\begin{Proposition}
Let $(M,c)$ be an even-dimensional conformal manifold admitting a twistor spinor $\varphi$. Then, at each point 
\begin{align}
\mathrm{Im}(\mathcal{O}) \subset{\cal  L}_{\varphi} \label{Ots}
\end{align}
 In particular, $\mathcal{O}$ vanishes if there are twistor spinors whose associated subspaces $L$ are transversal on an open and dense subset of $M$. 
 \end{Proposition}

\begin{proof}
Combining Theorem \ref{maintheo} with relation \eqref{holts} yields that
\begin{align}
s_- \cdot \mathcal{O}(X) \cdot \psi = 0 \label{spintractor}.
\end{align}
Filling in the technical details how $\psi$ is related to $\varphi$ by means of a metric in the conformal class as done in \cite{leitnerhabil}, reveals that \eqref{spintractor} is equivalent to
\begin{align}
\mathcal{O}(X) \cdot \varphi(x) = 0, \text{ for } \varphi(x) \neq 0,
\end{align}
which is clearly equivalent to \eqref{Ots}.
\end{proof}

We continue by  combining  Theorem \ref{obstrhol} with the results in Section \ref{adaptsec}.
In the non-generic case, i.e., when $\fhol(M,c)\not= \fso(p+1,q+1)$,  Proposition \ref{furr} shows that    the image of $\cal O$ is lightlike over an open dense set in $M$, and hence everywhere:
\begin{Corollary} \label{obstrrank2}
 If $\mathfrak{hol}(M,c)$ is not generic, then $\mathrm{Im}(\mathcal{O})$ is totally lightlike. In particular,   
if $(M,c)$ is Riemannian and  $\fhol(M,c)$ is not generic, then  $\mathcal{O} = 0$.
\end{Corollary}
The statement in Corollary \ref{obstrrank2} about Riemannian conformal structure can  be pieced together  from several results in the literature: the decomposition theorem in \cite{armstrong07conf} states that a conformal structure with  holonomy reduced from $\so(1,n+1)$, locally over an open dense subset of $M$, contains an Einstein metric or a certain product of Einstein metrics. Corollary \ref{einsteincor} and the results in \cite{GoverLeitner09} about products of Einstein metrics  then ensure that $(M,c)$ admits an ambient metric whose Ricci tensor vanishes to infinite order, and hence that  the obstruction tensor vanishes. Our proof of $\cal O=0$ for Riemannian non generic conformal classes in Corollary \ref{obstrrank2} is self-contained and does not make use of the results in the literature.

We consider now several options for the rank of $\cal O$. 
We obtain from 
Proposition \ref{openorbit}:
\begin{Corollary}

If $\Hol^0(M,c)$ has no  open orbit on the  lightcone $\mathcal{N} \subset \rr^{p+1,q+1}$, then $\rk(\cal O)\le 1$.
\end{Corollary}
Indeed, if $\Hol^0(M,c)$ has no  open orbit on the  lightcone $\mathcal{N} \subset \rr^{p+1,q+1}$, then by 
Proposition \ref{openorbit} the rank of $\cal O$ is $\le 1$ on an open dense set. Hence, the rank is $\le 1$ everywhere.

Again we refer to  \cite{Alt12} were conformal structures with a transitive and irreducible action of the conformal holonomy are classified.
Moreover,   Proposition \ref{lightlike} implies:
\begin{Corollary} \label{obstrrank1}
If the rank of  $\mathcal{O}$ is maximal at some point $x \in M$, then $\mathfrak{hol}(M,c)=\so(p+1,q+1)$  is generic. In particular, all parallel tractors are obtained from  the tractor metric $h$ only.
\end{Corollary}

Corollary \ref{obstrrank1} demonstrates that the ambient obstruction tensor $\mathcal{O}$ can also be interpreted as an obstruction to the existence of parallel tractors on $(M,c)$ of any type. Namely for such a tractor to exist, $\mathcal{O}$ needs to have a nontrivial kernel everywhere. 
We analyse this phenomenon in more detail in  by focusing on parallel tractor forms and the associated normal conformal Killing forms (see Section \ref{bspsec}). Proposition \ref{ncform} implies:
\begin{Corollary} \label{ncform-cor}
If $(M,c)$ admits a nc-Killing form $\alpha \in \Omega^k(M)$, then   $\mathrm{Im}(\mathcal{O}) \wedge \alpha = 0$.
\end{Corollary}

\begin{Corollary} \label{ncvf}
If $V$ is a normal conformal vector field for $(M,c)$, then  $\mathrm{Im}(\mathcal{O}) \subset \mathbb{R} V$ whenever $V \neq 0$. In particular, $\mathcal{O}$ vanishes if there is a normal conformal vector field that is not lightlike, or if the space of normal conformal vector fields has dimension greater than 1. 
\end{Corollary}

In particular, Corollary \ref{ncvf} applies to Fefferman conformal structures $(M,c)$ of signature $(2k+1,2r+1)$, i.e. $\Hol(M,c) \subset \mathbf{SU}(k+1,r+1)$. They admit a distinguished normal conformal Killing vector field $V_F$. Thus, 
\begin{align}
\text{Im }\mathcal{O} \subset \mathbb{R} V_F, \label{cfeff}
\end{align}
 for which an independent proof can be found in  \cite{GrahamHirachi08}. For the Lorentzian case, i.e. $k=0$, any additional holonomy reduction will force $\mathcal{O}$ to vanish.

\begin{Proposition}
Let $(M,c)$ be a Lorentzian conformal manifold of even dimension $n$ with $\fhol(M,c) \varsubsetneq \fsu(1,\frac{n}{2})$. Then $\mathcal{O} = 0$.
\end{Proposition}

\begin{proof}
From the classification of irreducibly acting subalgebras of $\fso(2,n)$ in \cite{Di-ScalaLeistner11} and the results in \cite{alt-discala-leistner14} it follows that $\fhol(M,c)$ has to act with an invariant subspace. If the holonomy representation fixes a non-degenerate subspace or a lightlike  line in $\mathbb{R}^{2,n}$ the result follows with Corollaries \ref{einsteincor} and \ref{splitnd-cor}. Otherwise, $\fhol(M,c)$ fixes a totally lightlike $2$-plane in $\mathbb{R}^{2,n}$ and  again Corollary \ref{splitnd-cor} applies. That is, there is (at least locally) a metric $g \in c$ admitting a recurrent and nowhere vanishing null vector field $U$, i.e. $\nabla^g U = \theta \otimes U$ for some 1-form $\theta$ and  $\mathrm{Im}(\mathcal{O}) \subset \mathbb{R} U$. Assume now that $\mathcal{O}$ is nonzero at some point. It follows from \eqref{cfeff} that there is an open subset of $M$ on which $\mathbb{R}V_F = \mathbb{R}U$. However, this contradicts the fact that the twist\footnote{ Recall that for a vector field $X \in \mathfrak{X}(M)$ its twist is the 3-form $\omega_X := dX^{\flat} \wedge X^{\flat}$. Clearly, the condition $d\omega_X = 0$ depends on $\mathbb{R}X$ only.} of $U$ is given by $\omega_U = \theta \wedge U^{\flat} \wedge U^{\flat} = 0$ but $\omega_{V_F} \neq 0$ (see \cite{BaumLeitner04}). Thus, $\mathcal{O} \equiv 0$.
\end{proof}

\bbem
In similar fashion, Fefferman spaces over quaternionic contact structures, cf. \cite{alt08}, admit 3 linearly independent $\Hol(M,c)$-invariant almost complex structures which descend to pointwise linear independent nc-vector fields (or 1-forms) on $M$. Thus $\mathcal{O} \equiv 0$ for this case by Corollary \ref{ncvf}.
\ebem

\subsection{The obstruction tensor for Bryant conformal structures} \label{spin43}

We will now specialise to Bryant conformal structures in signature $(3,3)$ induced by a generic 3-distribution $\cal D \subset TM$ as in Section \ref{bspsec},  and deduce several new results about the relation of the generic distribution $\cal D$ and the image of $\cal O$.

Every Bryant conformal structure admits (and is equivalently characterised by) a parallel tractor 4-form $\widehat{\alpha} \in \Gamma(M, \Lambda^4 \mathcal{T})$ whose stabiliser under the  $\SO(4,4)$-action at each point is isomorphic to $\Spin(4,3) \subset \SO(4,4)$. In particular, $\Hol(M,c) \subset \Spin(4,3)$. For a fixed metric $g \in c$ and the corresponding splitting \eqref{splitncform}, i.e.,
\begin{align}
 \widehat{\alpha} = s_+^{\flat} \wedge \alpha + \alpha_0 +...,\label{alpha}
 \end{align}
one finds that $\alpha = l^{\flat}_1 \wedge l^{\flat}_2 \wedge l^{\flat}_3$ for $l_{i=1,2,3}$ some basis of $\cal D$ and $\alpha$ transforms conformally covariant under a change of $g$. Using this, we can derive constraints on the obstruction tensor for Bryant conformal structures.
 
 As an immediate
 consequence of Proposition \ref{ncform} and Corollary \ref{ncform-cor} we obtain:
\begin{Corollary}\label{bryantcor1}
Let $(M, c_{\cal D})$ be a Bryant conformal structure induced by a generic 3-distribution  $\cal D \subset TM$. Then $\mathcal{E} \subset \cal D$, and in particular,  $\mathrm{Im}(\mathcal{O}) \subset\cal  D$. 
\end{Corollary}

Moreover:
\begin{Corollary}\label{bryantcor2}
If  $\mathfrak{hol}(M,c) = \mathfrak{spin}(4,3)$, then $\cal D = \mathcal{E}$ everywhere on $M$.
\end{Corollary}
\bprf
The Lie algebra $\mathfrak{spin}(4,3)$ equals the stabiliser algebra of a spinor $\psi$ of nonzero length in signature $(4,4)$ which corresponds via some $g \in c$ to a twistor spinor $\varphi$ with $L_{\varphi } = \cal D$ at every point (see Section \ref{bspsec}). Thus, $(s_-^{\flat} \wedge l^{\flat} ) \cdot \psi = 0$ for every $l \in \cal D$, i.e. $\cal D \subset \mathcal{E}$. 
\eprf
\bbem
This agrees with the curved orbit decomposition from \cite{cgh}, cf. the  discussion in Section \ref{cosec} for this particular case. Indeed, as discussed in \cite{cgh} for the general case, the curved orbits correspond to the $\Spin(4,3)$-orbits on $\SO(4,4)/\mathrm{Stab}_{\SO(4,4)}(\cal L)$, where $\cal L \subset \rr^{4,4}$ is a null line. However, there is only one such orbit as $\Spin(4,3)$ acts transitively on the projectivised lightcone in $\mathbb{R}^{4,4}$. 
\ebem

\begin{Proposition} \label{spincase}
Assume that $\fhol(M,c)\varsubsetneq \spin(4,3) \subset \fso(4,4)$. Then  $\rk(\mathcal{O}) \leq 1$.
\end{Proposition}

\bprf
Suppose first that there is an open set $U \subset M$ on which $\mathcal{E}$ has dimension 3, i.e., by Corollary \ref{bryantcor1} we have  $\mathcal{E} = \cal D$ over $U$. By passing to a subset of $U$ if necessary, we may assume that $U$ is a $\mathcal{E}$-adapted open set. Let $V_{i=1,2,3}$ be a pointwise basis of $\mathcal{E}$ over $U$ depending smoothly on $x$. Let $V'_i$ be lightlike vector fields on $U$ such that $g(V_i,V_j') = \delta_{ij}$. We have seen that in this case  the 15 elements in \eqref{15elements}
are pointwise linearly independent in $\mathfrak{hol}(M,c) \cap \mathfrak{p}$. But then it follows immediately from Proposition \ref{openorbit} that dim $\mathfrak{hol}(M,c) \geq 15+6 = 21$, which is the dimension of $\mathfrak{spin}(4,3)$. Thus $\fhol(M,c)$ is no proper subalgebra of $\fso(4,4)$.

We have to conclude that the set on which $r^\cal E \leq 2$ is open and dense in $M$. In particular,  $\mathrm{rk}(\mathcal{O}) < 3$ on an open and dense subset of $M$. However, the set on which  $\rk(\mathcal{O}) < 3$ is also closed and since $M$ is connected it follows that  $\rk(\mathcal{O}) < 3$ on $M$.
Assume next that there is $x \in M$ such that  $\rk(\mathcal{O}) = 2$ at $x$. Since the subset on which  $\rk(\mathcal{O}) \geq 2$ is open in $M$ it follows that  $\mathrm{rk}(\mathcal{O}) = 2 $ on some open set $U$ of $M$. After restricting $U$ we may assume that $U$ is a $\mathcal{E}$-adapted open set and $r^\cal E = 2$ on $U$. Thus, $\mathcal{E}$ is over $U$ a 2-dimensional subbundle of $\cal D$. By Theorem \ref{eint}, $\mathcal{E}$ is integrable over $U$ which contradicts $\cal D$ being generic. Consequently,  $\rk(\mathcal{O}) \leq 1$ everywhere.
\eprf

\begin{bsp}
Proposition \ref{spincase} applies to the situation when $\Hol(M,c)$ lies in the intersection of $Spin(4,3)$ with the stabiliser of a totally degenerate subspace $\cal H \subset \mathbb{R}^{4,4}$. For dim $\cal H=4$, this intersection is isomorphic to

\[\spin(3,4)_{\cal H}=\left\{
\begin{pmatrix}
Z
&X\\0&-Z^\top \end{pmatrix}\mid Z\in \csp_2\rr,\ X\in \fso(4),\tr(X\J )=0 \right\},
\]
where 
\[\J=\begin{pmatrix}0&\mathbf{1}_2\\-\1_2&0\end{pmatrix},\]
and
\[\csp_2\rr=\left\{Z\in \fgl_4\R\mid Z^\top\J+\J Z-\tfrac{1}{2}\tr (Z)\J=0\right\}=\rr\, \1_4\+ \sp_2\rr.\]
Moreover, since 
the Lie group $\Spin(3,4)\subset \SO(4,4)$ corresponding to $\spin(3,4)\subset\so(4,4)$ acts transitively on triples
\[\{ (\s_+,\mathcal H, \s_-)\mid  \mathcal H\text{ a totally null $4$-plane, }\s_+\in \mathcal H, \s_-\in \rr^8\text{  null,  } g(\s_+,{\s_-})=1\},\]
we can express the stabiliser of $\cal H$ in conjunction with the $|1|$-grading $\spin(3,4)=\g_{-1}\+\g_0\+\g_1$
in a basis
$(\s_+,\e_a,\s_-,\e_\aa)$ for $a=1,2,3$ and $\aa=a+3$,
as
\[
\spin(3,4)_{\cal H}=
\left\{
\left(
\begin{array}{cc|cc}
r& w^\top &0&\bar{w}^\top
\\
v^\top & Z & -\bar{w}& X
\\
\hline
0&0& -r&-v
\\
0&0& -w&-Z^\top
\end{array}
\right)
\mid
\begin{array}{l}
w=(w^a)\in \R^3,\ \bar{w}=(\bar{w}^\aa)\in \rr^3,
v=(v_\bb)\in (\rr^3)^*,
\\
X=(X_{\ a}^{\bb})\in \so(3),\ Z=(Z_{\ a}^{b})\in \fgl_3\rr,\\[1mm]
w^{3}=Z_{\ 1}^{2},\ w^{1}=-Z_{\ 3}^{2},\ 
v_1 =-Z_{\ 2}^{3},\  v_3=Z_{\ 2}^{1},
\\[1mm]
r=Z_{\ 1}^{1}-Z_{\ 2}^{2}+Z_{\ 3}^{3},\ 
\bar{w}^2=-X_{\ 3}^{1}.
\end{array}
\right\}.
\]
Here $(r,Z,X)$ corresponds to the $\g_0$ part whereas $(w,\bar{w})$ correspond tot he $\g_{-1}$ and $v$ to the $\g_1$-part. In particular, the intersection $\fp_\cal H$ of $\spin(3,4)_{\cal H}$ with the parabolic $\mathfrak p$ is given by setting $w$ and $\bar{w}$ to zero, and the intersection $\cal E$ of $\spin(3,4)_{\cal H}$ with $\g_1$ by requiring in addition that $X=Z=r=0$. Note that $\cal E$ is one dimensional.

In regards to examples of this situation, we recall that
in \cite{AndersonLeistnerNurowski15}  a certain class of  Bryant's conformal structures was studied. They are defined by  a rank $ 3$ distribution ${\mathcal D}_f$ on 
$\rr^6$ with coordinates $(x^1,x^2,x^3,y^1,y^2,y^3)$ given by the annihilator of three 1-forms
$$\theta_1=d  y^1+x^2d  x^3,\quad\theta_2=d  y^2+fd  x^1,\quad\theta_3=d  y^3+x^1d  x^2,$$
where  $f=f(x^1,x^2,x^3)$ is a differentiable function of the variables $(x^1,x^2,x^3)$ only. It was shown that, whenever $f$ depends only on $x^3$ and $x^1$, the corresponding conformal class contains a metric for which the image of the Schouten tensor lies in a parallel rank $3$ distribution, which implies \cite{Lischewski15}  that 
the conformal holonomy is contained in $\spin(3,4)_{\cal H}$.  
In addition, these conformal structures turned out to have vanishing obstruction tensor, and therefore they admit ambient metrics.
For the conformal class defined by $\cal D_f$ with $f = x^1(x^3)^2$, an ambient metric with holonomy equal to $\spin(3,4)_{\cal H}$ was found, and for this example also the conformal holonomy is equal to $\spin(3,4)_{\cal H}$. 
\end{bsp}
\bbem
We  point out that there is a large class of  examples of Bryant conformal structures with  $f$ depending on  three variables $x^1,x^2,x^3$ for which the obstruction tensor has rank $3$, e.g., the one with $f=x^3+x^1 x^2+(x^2)^2+(x^3)^2$  in \cite{AndersonLeistnerNurowski15}. From our Proposition \ref{spincase} it follows that these examples have $\hol(M,c_{\cal D_f})=\spin(4,3)$.

More difficult is the question of finding examples with $\rk(\cal O)=1$. Of course, a general conformal structure with holonomy $\fsu(2,2)\subset \spin(4,3)$ has  $\rk(\cal O)=1$, but we 
are not aware of an explicit example with $\rk(\cal O)=1$ and   $\fhol(M,c_\cal D)\subset \spin(4,3)_{\cal H}$. Other examples with $\rk(\cal O)=1$, not necessarily with $\fhol(M,c_\cal D)\subset \spin(4,3)$, are given by pp-waves and their generalisation to arbitrary signature \cite{leistner-nurowski08,aipt2}.
\ebem

Finally, Theorem \ref{eint} implies:

\begin{Corollary}
Suppose $(M,c)$ is of signature $(3,3)$ and  $\rk(\mathcal{O}) \leq 3$ on some open set and  $\mathrm{Im}(\mathcal{O})$ is not integrable. Then $\mathfrak{hol}(M,c)$ is either equal to $\mathfrak{so}(4,4)$ or to $ \mathfrak{spin}(4,3)$.
\end{Corollary}

\bprf
From the assumptions,  $\rk(\mathcal{O}) \geq 2$ on an open set. If  $\rk(\mathcal{O}) = 2$ on an open set, it follows from Theorem \ref{eint} that $\mathcal{E}$ must have dimension at least 3 on this set. Otherwise the image of $\mathcal{O}$ would be integrable. But then the statement follows from  Theorem \ref{eint}. Otherwise the set on which  $\rk(\mathcal{O}) \geq 3$ is open and dense and the statement is an immediate consequence of Theorem~\ref{eint}.
\eprf




\small
\bibliographystyle{abbrv}
\bibliography{literatur2,geobib}

\end{document}